\RequirePackage{fix-cm}
\documentclass[smallextended]{svjour3}       
\smartqed  
\usepackage{graphicx}
\usepackage{amsmath}
\usepackage{amssymb}
\usepackage{algorithm}
\usepackage{algorithmic}[1]
\usepackage{enumerate}
\usepackage{graphicx}
\usepackage{epstopdf}
\usepackage{subcaption}
\usepackage{multirow}
\usepackage{enumitem}
\usepackage{color}
\usepackage{url}

\usepackage{amsthm}
\newcommand{\limto}{\rightarrow}

\newcommand{\R}{\mathbb R}

\newcommand{\argmin}{\mathop{\rm argmin}}

\newcommand{\cf}{{\it cf.}}
\newcommand{\eg}{{\it e.g.}}
\newcommand{\ie}{{\it i.e.}}

\newtheorem{thm}{Theorem}
\newtheorem{coro}{Corollary}

\newtheorem{prop}{Proposition}

\newtheorem{ass}{Assumption}
\newtheorem{fact}{Fact}

\newcounter{spb}
\newcommand{\subpb}{(\alph{spb}) \addtocounter{spb}{1}}

\newcommand{\resetspb}{\setcounter{spb}{1}}

\begin{document}

\title{A Family of Inexact SQA Methods for Non-Smooth Convex Minimization with Provable Convergence Guarantees Based on the Luo-Tseng Error Bound Property
\thanks{This research is supported in part by the Hong Kong Research Grants Council (RGC) General Research Fund (GRF) Projects CUHK 14206814 and CUHK 14208117 and in part by a gift grant from Microsoft Research Asia.}
}

\titlerunning{Inexact SQA Methods and Error Bound-Based Convergence Analysis}        

\author{Man-Chung Yue         \and
        Zirui Zhou		\and
        \mbox{Anthony Man-Cho So}
}

\authorrunning{M.-C. Yue, Z. Zhou, A. M.-C. So} 

\institute{Man-Chung Yue \at
              Imperial College Business School \\
              Imperial College London, United Kingdom \\
              \email{m.yue@imperial.ac.uk}           
           \and
           Zirui Zhou \at
           Department of Mathematics \\
           Simon Fraser University, Canada \\
              \email{ziruiz@sfu.ca}
           \and
           Anthony Man-Cho So \at
           Department of Systems Engineering and Engineering Management \\
              The Chinese University of Hong Kong, Shatin, N. T., Hong Kong\\
           	  \email{manchoso@se.cuhk.edu.hk}
}

\date{Received: date / Accepted: date}

\maketitle

\begin{abstract}
We propose a new family of inexact sequential quadratic approximation (SQA) methods, which we call the \emph{inexact regularized proximal Newton} (\textsf{IRPN}) method, for minimizing the sum of two closed proper convex functions, one of which is smooth and the other is possibly non-smooth. Our proposed method features strong convergence guarantees even when applied to problems with degenerate solutions while allowing the inner minimization to be solved inexactly. Specifically, we prove that when the problem possesses the so-called Luo-Tseng error bound (EB) property, \textsf{IRPN} converges globally to an optimal solution, and the local convergence rate of the sequence of iterates generated by \textsf{IRPN} is linear, superlinear, or even quadratic, depending on the choice of parameters of the algorithm. Prior to this work, such EB property has been extensively used to establish the linear convergence of various first-order methods. However, to the best of our knowledge, this work is the first to use the Luo-Tseng EB property to establish the superlinear convergence of SQA-type methods for non-smooth convex minimization. As a consequence of our result, \textsf{IRPN} is capable of solving regularized regression or classification problems under the high-dimensional setting with provable convergence guarantees. We compare our proposed \textsf{IRPN} with several empirically efficient algorithms by applying them to the $\ell_1$-regularized logistic regression problem. Experiment results show the competitiveness of our proposed method.
\end{abstract}

\section{Introduction}
\label{sec:intro}
A wide range of tasks in machine learning and statistics can be formulated as a non-smooth convex minimization problem of the form\footnote{Some authors refer to this as a convex composite minimization problem.}
\begin{equation}\label{eq:basic-prob}
\min_{x\in\mathbb{R}^n} \left\{ F(x) := f(x) + g(x) \right\},
\end{equation}
where $f,g:\mathbb{R}^n\rightarrow(-\infty,+\infty]$ are closed proper convex functions with $f$ being 
twice continuously differentiable on an open subset of $\mathbb{R}^n$ containing the effective domain $\mbox{dom}(g)$ of $g$ and $g$ being non-smooth.
A popular choice for solving problem \eqref{eq:basic-prob} is the sequential quadratic approximation (SQA) method (also called the proximal Newton method). Roughly speaking, in iteration $k$ of a generic SQA method, one computes an (approximate) minimizer $\hat{x}^{k+1}$ of a quadratic model $q_k$ of the objective function $F$ at $x^k$, where
\begin{equation}\label{eq:quad-approx}
q_k(x) := f(x^k) + \nabla f(x^k)^T(x - x^k) + \frac{1}{2}(x - x^k)^TH_k(x-x^k) + g(x)
\end{equation}
and $H_k$ is a positive definite matrix approximating the Hessian $\nabla^2 f(x^k)$ of $f$ at $x^k$. A step size $\alpha_k$ is obtained by performing a backtracking line search along the direction $d_k:=\hat{x}^{k+1}-x^k$, and then $x^{k+1} := x^k + \alpha_k d^k$ is returned as the next iterate. Since in most cases, the (approximate) minimizer $\hat{x}^{k+1}$ of~\eqref{eq:quad-approx} does not admit a closed-form expression, an iterative algorithm, which we shall refer to as the \emph{inner solver} in the sequel, is invoked. Throughout the paper, we will call the problem of minimizing the quadratic model~\eqref{eq:quad-approx} the \emph{inner problem} and its solution the \emph{inner solution}. We will also refer to the task of minimizing~\eqref{eq:quad-approx} as \emph{inner minimization}.

There are three important ingredients that, more or less, determine an SQA method: the approximate Hessian $H_k$, the inner solver for minimizing $q_k$, and the stopping criterion of the inner solver to control the inexactness of the approximate inner solution $\hat{x}^{k+1}$. Indeed, many existing SQA methods and their variants that are tailored for special instances of problem~\eqref{eq:basic-prob} can be obtained by specifying the aforementioned ingredients. Friedman et al.~\cite{FHHT07} developed the \textsf{GLMNET} algorithm for solving the $\ell_1$-regularized logistic regression, where $H_k$ is set to be the exact Hessian $\nabla^2f(x^k)$ and a coordinate minimization method is used as the inner solver. Yuan et al.~\cite{YHL12} improved \textsf{GLMNET} by replacing $H_k$ with $\nabla^2f(x^k) + \nu I$ for some constant $\nu>0$ and adding a heuristic adaptive stopping strategy for inner minimization. This algorithm, called \textsf{newGLMNET}, is now the workhorse of the well-known LIBLINEAR package~\cite{FCHWL08} for large-scale linear classification. Hsieh et al.~\cite{HDRS11} proposed the \textsf{QUIC} algorithm for solving sparse inverse covariance matrix estimation, which makes use of a quasi-Newton model to form the quadratic model~\eqref{eq:quad-approx}. Similar strategies have also been employed in Olsen et al.~\cite{ONRO12}, where the inner problem is solved by the fast iterative soft-shrinkage algorithm (FISTA)~\cite{beck2009fast}. Other SQA variants can be found, \eg, in \cite{SBFM09,BF12,ZYDR14}.

Although the numerical advantage of the above algorithms has been well documented, their global convergence and convergence rate guarantees require the inner problems to be solved exactly. Such requirement is rarely satisfied in practice. To address this issue, Lee et al.~\cite{LSS14} and Byrd et al.~\cite{BNO16} proposed several families of SQA methods along with stopping criteria for the inner problem and showed that they are globally convergent and have a local superlinear convergence rate. However, all these convergence guarantees require the strong convexity of $f$. In fact, the strong convexity property is needed not only in the analyses but also for the well-definedness of the proposed methods.\footnote{For instance, the exact Hessian $H_k=\nabla^2 f(x^k)$ is used in~\cite{LSS14}. If $f$ is not strongly convex, then neither is the quadratic model \eqref{eq:quad-approx}. As such, the inner problem can have multiple minimizers and the next iterate $x^{k+1}$ is not well defined.}
Unfortunately, such a property is absent in many applications of interest. For example, consider the $\ell_1$-regularized least squares regression problem, whose objective function takes the form
\begin{equation*}
F(x) = \frac{1}{m}\sum_{i=1}^m \left( a_i^Tx - b_i \right)^2 + \mu\|x\|_1, \quad \mu>0.
\end{equation*}
Observe that the smooth part $x \mapsto \tfrac{1}{m}\sum_{i=1}^m \left( a_i^Tx - b_i \right)^2$ is strongly convex if and only if the data matrix $A:=[a_1 \, \cdots \, a_m]\in\mathbb{R}^{n\times m}$ has full row rank. The latter cannot be ensured in most applications and is even impossible in the high-dimensional (\ie, $n\gg m$) setting. 
In addition, the authors of~\cite{LSS14} did not provide any global convergence result for the inexact version of their proposed method, while those of~\cite{BNO16} considered only the case where $g(x) = c\cdot \|x\|_1$ for some constant $c>0$. On another front, Scheinberg and Tang~\cite{ST16} proposed an inexact SQA-type method and analyzed its global complexity. Although their analysis does not require the strong convexity of $f$, it only yields a sublinear convergence rate for the proposed method. In view of the above discussion, the theoretical study of existing SQA methods is rather incomplete, and we are motivated to develop an algorithm for solving problem~\eqref{eq:basic-prob} that does not require an exact inner solver or the strong convexity of $f$ but can still be shown to converge globally and possess a local superlinear convergence rate. 

In this paper, we propose a new family of inexact SQA methods, which we call the \emph{inexact regularized proximal Newton} (\textsf{IRPN}) method, that are capable of solving non-strongly convex instances of problem \eqref{eq:basic-prob} without using an exact inner solver. In iteration $k$, \textsf{IRPN} takes $H_k$ to be the regularized Hessian $H_k = \nabla^2 f(x^k) + \mu_kI$, where $\mu_k = c\cdot r(x^k)^{\rho} > 0$ for some constants $c>0$, $\rho\in[0,1]$, and $r$ is an easily computable residual function that measures the proximity of $x^k$ to the optimal solution set (see Section \ref{sec:error-bound} for details). The inner problem is solved up to an accuracy that is determined by an adaptive inexactness condition, which also depends on the residue $r(x^k)$ and the parameter $\rho$. It is worth noting that the idea of regularization is not new for second-order methods. Indeed, Li et al. \cite{LFQY04} investigated the regularized Newton method for solving smooth convex minimization problems with degenerate solutions. The well-known Levenberg-Marquardt (LM) method for solving nonlinear equations is essentially a regularized Gauss-Newton method~\cite{M78,YF01}. The algorithm that is closest in spirit to ours is \textsf{newGLMNET} \cite{YHL12}, which aims at solving problems of the form \eqref{eq:basic-prob} and also uses a regularized Hessian, albeit the regularization parameter remains constant throughout the entire execution; \ie, $\mu_k = \nu$ for all $k$ with $\nu>0$. Hence, it may be seen as a special case ($\rho = 0$) of \textsf{IRPN}. However, the constant $\nu$ is chosen empirically and a heuristic stopping rule is adopted for the inner solver.

Another motivation of this research comes from the seminal paper \cite{LT93} of Luo and Tseng, in which a certain error bound (EB) property was introduced. The Luo-Tseng EB property is provably a less stringent requirement on the objective function $F$ of problem~\eqref{eq:basic-prob} than strong convexity and has been shown to hold for a wide range of $F$'s; see Fact~\ref{fact:eb-holds}. Since its introduction, the Luo-Tseng EB property has played an important role in the convergence rate analysis of various first-order methods. Specifically, it has been utilized to establish the linear convergence of the projected gradient descent, proximal gradient descent, coordinate descent, and block coordinate descent methods for solving possibly non-strongly convex instances of problem~\eqref{eq:basic-prob}; see~\cite{LT93,tseng2009coordinate,tseng2010approximation,HZSL13,ZJL13}. It is interesting to note that if the objective function $F$ of problem~\eqref{eq:basic-prob} satisfies the Luo-Tseng EB property, then it also satisfies the so-called \emph{Kurdyka-{\L}ojasiewicz (KL)} property with exponent $1/2$~\cite{LP17}.\footnote{In~\cite{karimi2016linear} the authors considered global versions of the Luo-Tseng EB and KL properties and showed that they are equivalent. However, none of the scenarios listed in Fact~\ref{fact:eb-holds} except (S1) are known to possess the global Luo-Tseng EB property stated in~\cite{karimi2016linear}.} The latter can also be used to establish the linear convergence of various first-order methods; see, \eg,~\cite{LP17,LSW17} and the references therein.

Following the great success of EB-based analysis of first-order methods, it is natural to ask whether one can develop an EB-based analysis of second-order methods for solving problem \eqref{eq:basic-prob}. Such an analysis is highly desirable, as it would advance our understanding of the strength and limitation of SQA-type methods and more accurately capture the interplay between the algorithmic and geometric properties of problem~\eqref{eq:basic-prob}. It turns out that some partial answers in this direction are available. Li et al. \cite{LFQY04} proposed a regularized Newton method along with inexactness conditions for solving \emph{smooth} convex minimization problems and proved its global convergence and local quadratic convergence based on an EB property. Unfortunately, their analysis relies heavily on the eigenvalue perturbation properties of the Hessian of the smooth objective function and hence cannot be directly extended to second-order methods for solving the \emph{non-smooth} minimization problem~\eqref{eq:basic-prob}. Tseng and Yun \cite{tseng2009coordinate} proposed the (block) coordinate gradient descent (CGD) method, which includes the SQA method as a special case, for solving problem~\eqref{eq:basic-prob} with non-convex $f$ and non-smooth convex $g$ that has a separable structure. Although their analysis utilizes the Luo-Tseng EB property, their proposed CGD method requires the sequence of approximate Hessians $\{H_k\}_{k\ge0}$ to be uniformly lower bounded---\ie, $\lambda_{\min}(H_k)\ge \underline{\lambda}$ for all $k$, where $\underline{\lambda}>0$ is some constant---and the provable local convergence rate is only linear. Another important contribution along this direction is the work~\cite{Fis02} by Fischer, in which he studied an abstract iterative framework for solving a class of generalized equations and proved that it has a local superlinear convergence rate. Besides requiring an EB-type property and an inexactness condition, the result requires the distance between two consecutive iterates to be linearly bounded by the distance between the current iterate and the set of solutions $\mathcal{X}$ to the generalized equation; \ie, $\|x^{k+1} - x^k\|=O(\mbox{dist}(x^k,\mathcal{X}))$. However, for a concrete algorithm, establishing such an estimate is usually a non-trivial task. Moreover, the iterative framework in \cite{Fis02} does not involve regularization (see displayed equations~(3) and~(4) of~\cite{Fis02}). Hence, it is not obvious how our proposed method can be put into that framework. Recently, the authors of~\cite{FFH13,facchinei2014lp,fischer2016globally} studied an inexact Newton-type method for solving constrained systems of equations and proved its global convergence and local quadratic convergence based on an EB property. 
As the set of optimal solutions to problem~\eqref{eq:basic-prob} can be characterized as the set of solutions to a certain system of equations (see Fact~\ref{fact:opt-fix-pt} and the discussion following it), their algorithm can also be applied to solve problem~\eqref{eq:basic-prob}. However, their method is fundamentally different from ours, as the inner problem of their method involves solving a linear program, while that of ours involves solving an unconstrained strongly convex minimization problem (see Section~\ref{sec:algorithm} for details). Interestingly, in the context of problem~\eqref{eq:basic-prob}, the EB property considered by them also coincides with ours. Nevertheless, the convergence analysis of their proposed method assumes that the linear program defining the inner problem satisfies a so-called uniform inf-boundedness property in a neighborhood of an optimal solution to problem~\eqref{eq:basic-prob}; see~\cite[Assumption 3]{facchinei2014lp}. Although several sufficient conditions for such assumption to hold are developed in~\cite[Section 3]{facchinei2014lp}, they do not apply to some of the instances of problem~\eqref{eq:basic-prob} that we are interested in.


Our contributions in this paper can be summarized as follows. First, we propose a new family of inexact SQA methods called \textsf{IRPN} for solving the non-smooth convex minimization problem~\eqref{eq:basic-prob}. The proposed method has a parameter $\rho\in[0,1]$ and comes with a stopping criterion for inner minimization that allows us to solve the inner problem inexactly. Second, we establish the global convergence and local convergence rate of \textsf{IRPN}. More precisely, we prove that each accumulation point of the sequence of iterates generated by \textsf{IRPN} is optimal for problem~\eqref{eq:basic-prob}. Furthermore, when problem~\eqref{eq:basic-prob} possesses the Luo-Tseng EB property, the sequence converges to an optimal solution, and the local convergence rate is at least \emph{$R$-linear} if $\rho=0$, \emph{$Q$-superlinear} if $\rho\in(0,1)$, and \emph{$Q$-quadratic} if $\rho = 1$ (see~\cite[Appendix A.2]{NW06} for the definitions of these different types of convergence). Our analysis is novel and establish, for the first time, the superlinear convergence of inexact SQA-type methods when applied to instances of problem \eqref{eq:basic-prob} that possess the Luo-Tseng EB property.

The paper is organized as follows. In Section \ref{sec:preliminaries}, we set the stage for our theoretical development by stating the basic assumptions and reviewing the relevant existing results. We also discuss the Luo-Tseng EB property in the context of problem~\eqref{eq:basic-prob} and some other related regularity conditions. In Section~\ref{sec:algorithm}, we detail our proposed \textsf{IRPN}, including the specification of the regularized Hessian and inexactness condition for inner minimization. In Section~\ref{sec:convergence-analysis}, we establish the global convergence and local convergence rate of \textsf{IRPN}. We then present some numerical results in Section~\ref{sec:numerical}. Finally, we conclude our paper in Section \ref{sec:conclusion}.

\emph{Notation.} We denote the optimal value of and the set of optimal solutions to problem \eqref{eq:basic-prob} by $F^{*}$ and $\mathcal{X}$, respectively. We use $\|\cdot\|$ to denote the Euclidean norm for vectors and the operator norm (see~\cite[Exercise I.2.4]{B97}) for matrices. Given a set $C\subseteq \mathbb{R}^n$, we denote by $\mbox{dist}(x, C)$ the distance from $x$ to $C$; \ie, $\mbox{dist}(x,C) = \inf\{\|x - u\|\mid u\in C\}$. For a closed convex function $h$, we denote by $\partial h$ the subdifferential of $h$. If $h$ is continuously differentiable (resp. twice continuously differentiable), we denote by $\nabla h$ (resp. $\nabla^2 h$) the gradient (resp. Hessian) of $h$.

\section{Preliminaries}\label{sec:preliminaries}
Consider the convex minimization problem~\eqref{eq:basic-prob}. We assume that $f,g$ are closed proper convex functions with $f$ being twice continuously differentiable on an open subset of ${\mathbb R}^n$ containing ${\rm dom}(g)$ and $g$ being possibly non-smooth. To avoid ill-posed problems, we assume that the set of optimal solutions $\mathcal{X}$ to problem~\eqref{eq:basic-prob} is non-empty, so that the optimal value $F^*$ of problem~\eqref{eq:basic-prob} is finite. Furthermore, we make the following assumptions.
\begin{ass}\label{ass:smoothness}
The function $f$ in problem \eqref{eq:basic-prob} satisfies the following:
\begin{enumerate}[wide = 0pt, labelwidth = 2em, labelsep*=0em, itemindent = 0pt, leftmargin = \dimexpr\labelwidth + \labelsep\relax, noitemsep,topsep = 1ex, font=\normalfont, label=(\alph*)]
\item The gradient $\nabla f$ is Lipschitz continuous on an open set $\mathcal{U}$ containing ${\rm dom}(g)$; \ie, there exists a constant $L_1>0$ such that
$$
\|\nabla f(y) - \nabla f(z)\| \leq L_1\|y - z\| \quad\forall y,z\in\mathcal{U}.
$$
\label{ass:L1}

\item The Hessian $\nabla^2f$ is Lipschitz continuous on an open set $\mathcal{U}$ containing ${\rm dom}(g)$; \ie, there exists a constant $L_2>0$ such that
$$
\|\nabla^2 f(y) - \nabla^2 f(z)\| \leq L_2\|y - z\| \quad\forall y,z\in\mathcal{U}.
$$
\label{ass:L2}
\end{enumerate}
\end{ass}
\noindent The above assumptions are standard in the analysis of Newton-type methods. As we will see in Sections 4 and 5, Assumption~\ref{ass:smoothness}\ref{ass:L1} is crucial to the global convergence analysis of \textsf{IRPN}, while Assumption~\ref{ass:smoothness}\ref{ass:L2} is needed for the local convergence analysis. The following result is a direct consequence of Assumption~\ref{ass:smoothness}; see, \eg,~\cite[Lemma 1.2.2]{nesterov2004introductory}.
\begin{fact}\label{fact:bound-hess}
Under Assumption \ref{ass:smoothness}\ref{ass:L1}, for any $x\in{\rm dom}(g)$, we have
$$\lambda_{\max}(\nabla^2 f(x)) \leq L_1.$$ 
\end{fact}

\subsection{Optimality Conditions and Residual Functions}\label{sec:opt-resi-map}
We now introduce the optimality condition of problem~\eqref{eq:basic-prob}. Given a closed proper convex function $h:{\mathbb R}^n \rightarrow (-\infty,+\infty]$, the so-called \emph{proximal operator} ${\rm prox}_h:{\mathbb R}^n \rightarrow {\mathbb R}^n$ of $h$ is defined by
$$
\mbox{prox}_h(v) := \argmin_{u\in\mathbb{R}^n} \left\{ \frac{1}{2}\|u - v\|^2 + h(u) \right\}.
$$
The proximal operator is the building block of many first-order methods for solving problem~\eqref{eq:basic-prob}, including the proximal gradient method and its accelerated versions~\cite{parikh2014proximal} and (block) coordinate gradient descent methods~\cite{tseng2009coordinate}. 
Moreover, the proximal operators of many non-smooth functions, such as the indicator function of a closed convex set, the $\ell_1$-norm, the grouped LASSO regularizer, the elastic net regularizer, and the nuclear norm, have closed-form representations; see, \eg,~\cite[Chapter 6]{parikh2014proximal}. 
As is well known, the set of optimal solutions to problem~\eqref{eq:basic-prob} can be characterized using the proximal operator. The following result, which can be found, \eg, in~\cite[Proposition 3.1]{CW05}, will play an important role in our subsequent development.
\begin{fact}\label{fact:opt-fix-pt}
A vector $x\in{\rm dom}(g)$ is an optimal solution to problem~\eqref{eq:basic-prob} if and only if for any $\tau>0$, 
$$
x = {\rm prox}_{\tau g}(x - \tau\nabla f(x)). 
$$
\end{fact}
\noindent Let $R:\mbox{dom}(g)\rightarrow\mathbb{R}^n$ be the map given by
\begin{equation}\label{eq:def-resi-map}
R(x) := x - {\rm prox}_g(x - \nabla f(x)).
\end{equation}
Fact~\ref{fact:opt-fix-pt} suggests that we can take $r(x) = \|R(x)\|$ as a \emph{residual function} for problem~\eqref{eq:basic-prob}; \ie, the function $r$ satisfies $r(x)\geq 0$ for all $x\in\mbox{dom}(g)$ and $r(x) = 0$ if and only if $x\in\mathcal{X}$. In addition, the following proposition shows that both $R$ and $r$ are Lipschitz continuous on ${\rm dom}(g)$ if Assumption~\ref{ass:smoothness}\ref{ass:L1} holds.
\begin{prop}\label{prop:lip-resi-map}
Suppose that Assumption \ref{ass:smoothness}\ref{ass:L1} holds. Then, for any $y,z\in{\rm dom}(g)$, we have
\begin{equation*}
|r(y) - r(z)| \leq \|R(y) - R(z)\| \leq (L_1 + 2)\|y - z\|.
\end{equation*}
\end{prop}
\begin{proof}
The first inequality is a direct consequence of the triangle inequality. We now prove the second one. Using the definition of $R$, we compute
\begin{equation*}
\begin{split}
\|R(y) - R(z)\| & = \|y - \mbox{prox}_g(y - \nabla f(y)) - z + \mbox{prox}_g(z - \nabla f(z))\| \\
& \leq \|y - z\| + \|\mbox{prox}_g(y - \nabla f(y)) - \mbox{prox}_g(z - \nabla f(z))\| \\
& \leq 2\|y - z\| + \|\nabla f(y) - \nabla f(z)\| \\
& \leq (L_1 + 2)\|y - z\|,
\end{split}
\end{equation*}
where the second inequality follows from the non-expansiveness of the proximal operator (see, \eg,~\cite[Lemma 2.4]{CW05}) and the last is by Assumption~\ref{ass:smoothness}\ref{ass:L1}. 
\end{proof}

\subsubsection{Inner Minimization}\label{sec:opt-resi-map-im}
Recall that in the $k$-th iteration of an SQA method, one (approximately) minimizes the quadratic model $q_k$ in~\eqref{eq:quad-approx}. By letting 
\begin{equation*}
f_k(x) := f(x^k) + \nabla f(x^k)^T(x - x^k) + \frac{1}{2}(x - x^k)^TH_k(x - x^k),
\end{equation*}
the inner problem reads
\begin{equation}\label{eq:inner-min}
\min_{x\in\mathbb{R}^n} \left\{ q_k(x) := f_k(x) + g(x) \right\},
\end{equation}
which is also a convex minimization problem of the form~\eqref{eq:basic-prob} with the smooth part being quadratic. Therefore, both the optimality condition and residual function studied earlier for problem \eqref{eq:basic-prob} can be adapted to the inner problem~\eqref{eq:inner-min}. The following corollary is immediate from Fact~\ref{fact:opt-fix-pt} and the fact that $\nabla f_k(x) = \nabla f(x^k) + H_k(x - x^k)$.
\begin{coro}\label{cor:opt-fix-pt-im}
A vector $x\in{\rm dom}(g)$ is an optimal solution to problem~\eqref{eq:inner-min} if and only if for any $\tau>0$, 
\begin{equation*}
x = {\rm prox}_{\tau g}(x - \tau\nabla f_k(x) )= {\rm prox}_{\tau g}\left( (I - \tau H_k)x - \tau(\nabla f(x^k) - H_k x^k) \right).
\end{equation*}
\end{coro}
\noindent Similar to the map $R$ in~\eqref{eq:def-resi-map}, we define the map $R_k:\mbox{dom}(g) \rightarrow\mathbb{R}^n$ by
$$
R_k(x) :=x - {\rm prox}_g(x - \nabla f_k(x) ) = x - {\rm prox}_g\left( (I - H_k)x - (\nabla f(x^k) - H_k x^k) \right).
$$
Furthermore, parallel to the residual function $r$ defined for problem~\eqref{eq:basic-prob}, we can define the residual function for problem~\eqref{eq:inner-min} as $r_k(x) = \|R_k(x)\|$. Now, following the lines of the proof of Proposition \ref{prop:lip-resi-map}, we can easily show that both $R_k$ and $r_k$ are Lipschitz continuous. 
\begin{coro}\label{cor:lip-resi-map-im}
For any $y,z\in\mathbb{R}^n$, we have\footnote{Note that Assumption \ref{ass:smoothness}\ref{ass:L1} is not required for Corollary~\ref{cor:lip-resi-map-im} to hold; \cf~Proposition~\ref{prop:lip-resi-map}.}
$$ |r_k(y) - r_k(z)| \leq \|R_k(y) - R_k(z)\| \leq (\lambda_{\max}(H_k) + 2)\|y - z\|. $$
\end{coro}

\subsection{The Luo-Tseng Error Bound Property}\label{sec:error-bound}
A prevailing assumption in existing convergence analyses of SQA methods for solving problem \eqref{eq:basic-prob} is the strong convexity of the smooth function $f$~\cite{LSS14,BNO16}. However, such assumption is invalid in many applications (see the discussion in Section \ref{sec:intro}). Instead of assuming strong convexity, our analysis of \textsf{IRPN} is based on the following local EB property.
\begin{ass}\label{ass:ass-eb}
{\bf (Luo-Tseng EB Property)} Let $r$ be the residual function defined in Section \ref{sec:opt-resi-map}. For any $\zeta\geq F^{*}$, there exist scalars $\kappa>0$ and $\epsilon>0$ such that 
\begin{equation}\label{eq:def-eb}
{\rm dist}(x,\mathcal{X}) \leq \kappa\cdot r(x) \quad \mbox{whenever} \; F(x)\leq \zeta \mbox{ and } r(x)\leq \epsilon. 
\end{equation}
\end{ass}
\noindent EBs have long been an important topic and permeate in all aspects of mathematical programming~\cite{P97,FP03}. The Luo-Tseng EB property~\eqref{eq:def-eb} for problem~\eqref{eq:basic-prob} was studied and popularized in a series of papers by Luo and Tseng~\cite{luo1992error,LT92,LT93}.\footnote{A similar EB property has been studied by Pang~\cite{pang1987posteriori} for linearly constrained variational inequalities.} In particular, many useful subclasses of problem \eqref{eq:basic-prob} have been shown to possess the Luo-Tseng EB property. The early results (see, \eg,~\cite{LT93} and the references therein) mainly focus on the case where the non-smooth function $g$ has a polyhedral epigraph. It is only recently that the validity of the Luo-Tseng EB is established for instances of problem~\eqref{eq:basic-prob} with a non-polyhedral $g$~\cite{tseng2010approximation,ZZS15} or even with a non-polyhedral optimal solution set~\cite{ZS15}. We briefly summarize these results below and refer the readers to the recent work of Zhou and So~\cite{ZS15} for a more detailed review of the developments of the Luo-Tseng EB. 
\begin{fact}\label{fact:eb-holds}
For problem \eqref{eq:basic-prob},  the Luo-Tseng EB (Assumption \ref{ass:ass-eb}) holds in any of the following scenarios:
\begin{enumerate}[wide = 0pt, labelwidth = 2em, labelsep*=0em, itemindent = 0pt, leftmargin = \dimexpr\labelwidth + \labelsep\relax, noitemsep,topsep = 1ex, font=\normalfont, label={(S\arabic*)}]
\item (\cite[Theorem 4]{tseng2009coordinate}; \cf~\cite[Theorem 3.1]{pang1987posteriori}) $f$ is strongly convex, $\nabla f$ is Lipschitz continuous, and $g$ is a closed proper convex function.

\item (\cite[Lemma 7]{tseng2009coordinate}; \cf~\cite[Theorem 2.1]{LT92}) $f$ takes the form $f(x) = h(Ax) + \langle c,x\rangle$, where $A\in\mathbb{R}^{m\times n}$ and $c\in\mathbb{R}^n$ are given, $h:\R^m\limto(-\infty,+\infty)$ is a continuously differentiable function with $h$ being strongly convex and $\nabla h$ being Lipschitz continuous on any compact convex subset of $\mathbb{R}^n$, and $g$ has a polyhedral epigraph.


\item (\cite[Corollaries 1 and 2]{ZZS15}; \cf~\cite[Theorem 2]{tseng2010approximation} and~\cite[Theorem 1]{ZJL13}) $f$ takes the form $f(x)=h(Ax)$, where $A\in\R^{m\times n}$, $h:\R^m\limto(-\infty,+\infty)$ are as in scenario (S2), and $g$ is the $\ell_{1,p}$-norm regularizer with $p\in[1,2]\cup\{+\infty\}$ (\ie, $g(x) = \sum_{J\in\mathcal{J}}\omega_J\|x_J\|_p$, where $\mathcal{J}$ is a partition of the index set $\{1,\ldots,n\}$, $x_J\in\R^{|J|}$ is the sub-vector obtained by restricting $x\in\R^n$ to the entries in $J\in\mathcal{J}$, and $\omega_J\ge0$ is a given parameter; see~\cite{ZZS15} and the references therein for some background on the $\ell_{1,p}$-norm regularizer).

\item (\cite[Proposition 12]{ZS15}) $f$ takes the form $f(X) = h(\mathcal{A}(X)) + \langle C,X \rangle$, where $\mathcal{A}:\mathbb{R}^{n\times p} \rightarrow\mathbb{R}^m$ is a linear mapping, $C\in\mathbb{R}^{n\times p}$ is a given matrix, $h:\mathbb{R}^m\rightarrow(-\infty,+\infty)$ is as in scenario (S2), $g$ is the nuclear norm regularizer (\ie, $g(X)$ equals the sum of all the singular values of $X$; see~\cite{ZS15} and the references therein for some background on the nuclear norm regularizer), and there exists an $X^{*}\in\mathcal{X}$ such that the following strict complementary-type condition holds (here, ${\rm ri}$ denotes the relative interior):
$$ \mathbf{0} \in \nabla f(X^{*}) + {\rm ri}(\partial g(X^{*})). $$
\end{enumerate}
\end{fact}
\noindent Note that in scenarios (S2)--(S4), the assumptions on $h$ are the same and can readily be shown to be satisfied by 
$ h(y) = \frac{1}{2}\|y - b\|^2$,
which corresponds to least squares regression, and
$ h(y) = \sum_{i=1}^m\log(1 + e^{-b_iy_i})$ with $b\in\{-1,1\}^m$,
which corresponds to logistic regression. The assumptions on $h$ are also satisfied by loss functions that arise in maximum likelihood estimation (MLE) for learning conditional random fields~\cite{ZYDR14} and MLE under Poisson noise~\cite{SAT04}. It follows from Fact~\ref{fact:eb-holds} that many problems of the form \eqref{eq:basic-prob} possess the Luo-Tseng EB property, even though they may not be strongly convex.

Prior to this work, the Luo-Tseng EB property \eqref{eq:def-eb} has been used to establish the local linear convergence of a number of first-order methods for solving problem~\eqref{eq:basic-prob}; see the discussion in Section~\ref{sec:intro}.
It has also been used to prove that certain primal-dual interior-point path following methods converge superlinearly~\cite{tseng2000error}.  However, all the methods in question are quite different in nature from the SQA methods considered in this paper. In Section~\ref{sec:convergence-analysis}, we show how the Luo-Tseng EB property \eqref{eq:def-eb} can be used to establish the local superlinear convergence of our proposed \textsf{IRPN}, thus further demonstrating its versatility in convergence analysis. 

\subsection{Other Regularity Conditions}
Besides the Luo-Tseng EB property~\eqref{eq:def-eb}, other regularity conditions have been used to establish the superlinear convergence of SQA methods for solving problem \eqref{eq:basic-prob}. Let us briefly review some of those conditions here and explain why they are too stringent for the scenarios that we are interested in.

Yen et al.~\cite{YHRD14} and Zhong et al.~\cite{ZYDR14} introduced the \emph{constant nullspace strong convexity} (CNSC) property of a smooth function. They showed that when the smooth function $f$ possesses the CNSC property and the non-smooth function $g$ satisfies some other regularity conditions, the proximal Newton and proximal quasi-Newton methods in~\cite{LSS14} converge quadratically and superlinearly, respectively. Nonetheless, their convergence results are different from those obtained in this paper. Indeed, let $\{x^k\}_{k\ge0}$ be the sequence of iterates generated by the above methods and $z^k$ be the projection of $x^k$ onto the subspace associated with the CNSC property. The results in~\cite{YHRD14,ZYDR14} imply that the sequence $\{z^k\}_{k\ge0}$ converges to a point $z^*$ quadratically or superlinearly. However, the convergence rate of $\{x^k\}_{k\ge0}$ remains unclear. In fact, there is no guarantee that $\{x^k\}_{k\ge 0}$ or $\{z^k\}_{k\ge0}$ converges to an optimal solution to the problem. We also remark that neither the inexactness of the inner minimization nor the global convergence of the methods is addressed in \cite{YHRD14,ZYDR14}. 

Dontchev and Rockafellar~\cite{dontchev2009implicit} developed a framework for solving generalized equations, which coincides with the standard SQA method when specialized to the optimality condition $\mathbf{0}\in \nabla f(x) + \partial g(x)$ of problem \eqref{eq:basic-prob}. The corresponding convergence results require the set-valued mapping $\nabla f + \partial g$ to be either \emph{metrically regular} or \emph{strongly metrically sub-regular}; see~\cite[Chapter 6C]{dontchev2009implicit}. However, both of these regularity conditions are provably more restrictive than the Luo-Tseng EB property~\eqref{eq:def-eb}. In particular, they require that the problem at hand has a unique optimal solution, which is not satisfied by many instances of problem~\eqref{eq:basic-prob} that are of interest. For example, consider the following two-dimensional $\ell_1$-regularized least squares regression problem:
\begin{equation*}
\min_{x_1,x_2\in\mathbb{R}} \left\{ \frac{1}{2}(x_1 + x_2 - 2)^2 + |x_1| + |x_2| \right\}.
\end{equation*}
This problem possesses the Luo-Tseng EB property~\eqref{eq:def-eb}, as it belongs to scenario (S2) in Fact~\ref{fact:eb-holds}. However, it can be verified that the set-valued mapping associated with the optimality condition of this problem is neither metrically regular nor strongly metrically sub-regular.

\section{The Inexact Regularized Proximal Newton Method}\label{sec:algorithm}
We now describe in detail our proposed algorithm---the inexact regularized proximal Newton (\textsf{IRPN}) method. The algorithm takes as input an initial iterate $x^0\in\mathbb{R}^n$, a constant $\epsilon_0>0$ that controls the solution precision, constants $\theta\in(0,1/2)$, $\zeta\in(\theta,1/2)$, $\eta\in(0,1)$ that are used to specify the inexactness condition and  line-search parameters, and constants $c>0$, $\rho\in[0,1]$ that are used to form the regularized Hessians $\{H_k\}_{k\ge0}$. As we will see in Section \ref{sec:convergence-analysis}, the local convergence rate of \textsf{IRPN} largely depends on the choice of $\rho$. 

At the current iterate $x^k$, we first construct a quadratic approximation of $F$ at $x^k$ by
$$ q_k(x) := f(x^k) + \nabla f(x^k)^T(x - x^k) + \frac{1}{2}(x - x^k)^TH_k(x - x^k) + g(x), $$
where $H_k = \nabla^2 f(x^k) + \mu_kI$ with $\mu_k = c\cdot r(x^k)^{\rho}$ and $r$ is the residual function defined in Section~\ref{sec:opt-resi-map} for problem~\eqref{eq:basic-prob}. Since $f$ is convex and $\mu_k>0$, the matrix $H_k$ is positive definite for all $k$. Hence, the quadratic model $q_k$ is strongly convex and has a unique minimizer. However, since the exact minimizer does not admit a closed-form expression in most cases, an iterative algorithm, such as a coordinate minimization method, the coordinate gradient descent method, or the accelerated proximal gradient method, is typically called to find an \emph{approximate} minimizer $\hat{x}^{k+1}$ of the quadratic model $q_k$. To ensure that the method has the desired convergence properties, we require the vector $\hat{x}^{k+1}$ to satisfy
\begin{subequations} \label{eq:inexact-cond}
\begin{eqnarray}
r_k(\hat{x}^{k+1}) &\leq& \eta\cdot \min\{r(x^k), r(x^k)^{1+\rho}\}, \label{eq:inexact-cond-res} \\
\noalign{\smallskip}
q_k(\hat{x}^{k+1}) - q_k(x^k) &\leq& \zeta \left( \ell_k(\hat{x}^{k+1}) - \ell_k(x^k) \right), \label{eq:inexact-cond-quad}
\end{eqnarray}
\end{subequations}
where $r_k$ is the residual function defined in Section \ref{sec:opt-resi-map-im} for the inner problem~\eqref{eq:inner-min} and $\ell_k$ is the first-order approximation of $F$ at $x^k$:
$$ \ell_k(x) := f(x^k) + \nabla f(x^k)^T(x - x^k) + g(x). $$
We will show in Lemma~\ref{lem:exist-hatx} that the inexactness condition~\eqref{eq:inexact-cond} can be satisfied by vectors that are sufficiently close to the exact minimizer of $q_k$. After finding the approximate minimizer $\hat{x}^{k+1}$, we perform a backtracking line search along the direction $d^k := \hat{x}^{k+1} - x^k$ to obtain a step size $\alpha_k>0$ that can guarantee a sufficient decrease in the objective value. The algorithm then steps into the next iterate by setting $x^{k+1} := x^k + \alpha_k d^k$. Finally, we terminate the algorithm when $r(x^k)$ is less than the prescribed precision $\epsilon_0$. We summarize the details of \textsf{IRPN} in Algorithm~\ref{alg:IRPN}.

\begin{algorithm}
\begin{algorithmic}[1]
\STATE \textbf{Input:} initial iterate $x^0\in\mathbb{R}^n$, constants $\epsilon_0>0$, $\theta\in(0,1/2)$, $\zeta\in(\theta,1/2)$, $\eta\in(0,1)$, $c>0$, $\rho\in[0,1]$, and $\beta\in(0,1)$.
  \FOR[\texttt{outer iteration}]{$k = 0,1,2,\ldots$} 
  \STATE compute the value of the residue $r(x^k)$
  \STATE if $r(x^k)\leq \epsilon_0$, terminate the algorithm and return $x^k$
  \STATE form the quadratic model $q_k$ with $H_k = \nabla^2 f(x^k) + \mu_kI$ and $\mu_k = c\cdot r(x^k)^{\rho}$
	\smallskip
  \STATE \COMMENT{\texttt{inner iteration}} call an inner solver and find an approximate minimizer $\hat{x}^{k+1}$ of $q_k$ that satisfies 
\begin{equation*}
r_k(\hat{x}^{k+1}) \leq \eta\cdot \min\{r(x^k), r(x^k)^{1+\rho}\} \mbox{ and } q_k(\hat{x}^{k+1}) - q_k(x^k)\leq \zeta \left( \ell_k(\hat{x}^{k+1}) - \ell_k(x^k) \right) \tag{\ref{eq:inexact-cond}}
\end{equation*}  

  \STATE set the search direction $d^k := \hat{x}^{k+1} - x^k$ and find the smallest integer $i\ge0$ such that
 \begin{equation}\label{eq:descent-cond}
 F(x^k) - F(x^k + \beta^i d^k) \geq \theta \left( \ell_k(x^k) - \ell_k(x^k+\beta^i d^k) \right)
 \end{equation}

  \STATE set the step size $\alpha_k = \beta^i$ and the next iterate $x^{k+1} = x^k + \alpha_k d^k$
\ENDFOR
\end{algorithmic}
\caption{Inexact Regularized Proximal Newton (\textsf{IRPN}) Method}
\label{alg:IRPN}
\end{algorithm}

The inexactness condition~\eqref{eq:inexact-cond} is similar to that proposed in~\cite{BNO16}. However, the analysis in~\cite{BNO16} cannot be applied to study the convergence behavior of \textsf{IRPN}. Indeed, the global convergence analysis therein requires that $H_k\succeq \lambda I$ with $\lambda>0$ for all $k$, while the local convergence rate analysis requires $\nabla^2 f$ to be positive definite at the limit point $x^{*}$ of the sequence $\{x^k\}_{k\geq 0}$. 
However, for many of the instances of problem~\eqref{eq:basic-prob} that we are interested in, neither of these requirements are guaranteed to be satisfied.

One advantage of our proposed \textsf{IRPN} lies in its flexibility. Indeed, for SQA methods that use regularized Hessians, it is typical to let the regularization parameter $\mu_k$ be of the same order as the residue $r(x^k)$~\cite{LFQY04,QS06}; \ie, $\mu_k=c\cdot r(x^k)$. By contrast, we can adjust the order of $\mu_k$ according to the parameter $\rho$. Therefore, \textsf{IRPN} is a tunable family of algorithms parametrized by $\rho\in[0,1]$. As we will see in Sections~\ref{sec:convergence-analysis} and~\ref{sec:numerical}, the parameter $\rho$ plays a dominant role in the local convergence rate of \textsf{IRPN}. Even more, \textsf{IRPN} allows one to choose the inner solver. With the objective being the sum of a strongly convex quadratic function and a non-smooth convex function, each inner problem is a well-structured convex minimization problem. Hence, one can exploit the structure of $\nabla^2 f$ and $g$ to design special inner solvers to speed up the inner minimization.

Before we discuss the convergence properties of \textsf{IRPN}, we need to show that it is well defined. Specifically, we need to argue that the inexactness condition~\eqref{eq:inexact-cond} is always feasible and the line search procedure will terminate in a finite number of steps. We first present the following lemma, which ensures that in each iteration, the inexactness condition~\eqref{eq:inexact-cond} is satisfied by vectors that are close enough to the exact minimizer of the inner problem. The proof is identical to~\cite[Lemma 4.5]{BNO16}, but we present it here for completeness.
\begin{lemma}\label{lem:exist-hatx}
For any iteration $k$ of Algorithm~\ref{alg:IRPN}, the inexactness condition~\eqref{eq:inexact-cond} is satisfied by any vector $x$ that is sufficiently close to the exact minimizer $\tilde{x}^{k+1}$ of the quadratic model $q_k$.
\end{lemma}

\begin{proof}
Consider the iterate $x^k$ in the $k$-th iteration. We may assume that $x^k$ is not optimal for problem~\eqref{eq:basic-prob}, for otherwise the algorithm would have terminated. The continuity of $r_k$ and the fact that $r_k(\tilde{x}^{k+1}) = 0$ imply that condition~\eqref{eq:inexact-cond-res} is satisfied by any $x$ that is sufficiently close to $\tilde{x}^{k+1}$. Now, let us consider condition~\eqref{eq:inexact-cond-quad}. Since $q_k(x) = \ell_k(x) + \frac{1}{2}(x - x^k)^TH_k(x - x^k)$ and $\tilde{x}^{k+1}$ is the minimizer of $q_k$, we have $H_k(x^k - \tilde{x}^{k+1}) \in \partial\ell_k(\tilde{x}^{k+1})$. This leads to $x^k\neq \tilde{x}^{k+1}$, for otherwise we would have $\mathbf{0}\in\partial\ell_k(x^k)$, which would imply that $x^k$ is an optimal solution to problem~\eqref{eq:basic-prob}. Moreover, by the convexity of $\ell_k$, we have
\begin{equation}\label{eq:tri-ineq}
\ell_k(x^k) \geq \ell_k(\tilde{x}^{k+1}) + (\tilde{x}^{k+1} - x^k)^T H_k(\tilde{x}^{k+1} - x^k). 
\end{equation}
Since $H_k = \nabla^2 f(x^k) + \mu_kI$ for some $\mu_k>0$, $H_k$ is positive definite. This, together with the fact that $x^k\neq \tilde{x}^{k+1}$, implies that $\ell_k(x^k) > \ell_k(\tilde{x}^{k+1})$. Furthermore, we have
\begin{equation*}
\begin{split}
q_k(\tilde{x}^{k+1}) - q_k(x^k) & = \ell_k(\tilde{x}^{k+1}) + \frac{1}{2}(\tilde{x}^{k+1} - x^k)^T H_k(\tilde{x}^{k+1} - x^k) - \ell_k(x^k)  \\
& \leq \frac{1}{2} \left( \ell_k(\tilde{x}^{k+1}) - \ell_k(x^k) \right) \\
& < \zeta \left( \ell_k(\tilde{x}^{k+1}) - \ell_k(x^k) \right),
\end{split}
\end{equation*}
where the first inequality is due to~\eqref{eq:tri-ineq} and the last is due to $\zeta\in(\theta,1/2)$ and $\ell_k(x^k)>\ell_k(\tilde{x}^{k+1})$. Therefore, by the continuity of $q_k$ and $\ell_k$, condition~\eqref{eq:inexact-cond-quad} is satisfied by any $x$ that is close enough to $\tilde{x}^{k+1}$. 
\end{proof}

We next show that the backtracking line search is well defined.
\begin{lemma}\label{lem:exist-step-size}
Suppose that Assumption \ref{ass:smoothness}\ref{ass:L1} holds. Then, for any iteration $k$ of Algorithm~\ref{alg:IRPN}, there exists an integer $i\ge0$ such that the descent condition~\eqref{eq:descent-cond} is satisfied. Moreover, the step size $\alpha_k$ obtained from the line search strategy satisfies
$$ \alpha_k \geq \frac{\beta(1-\theta)\mu_k}{(1-\zeta)L_1}. $$
\end{lemma}

\begin{proof}
By definition, we have
$$ q_k(\hat{x}^{k+1}) - q_k(x^k) = \ell_k(\hat{x}^{k+1}) - \ell_k(x^k) + \frac{1}{2}(\hat{x}^{k+1} - x^k)^TH_k(\hat{x}^{k+1} - x^k). $$
Since the approximate minimizer $\hat{x}^{k+1}$ satisfies the inexactness condition~\eqref{eq:inexact-cond}, the above implies that
\begin{eqnarray}
\ell_k(x^k) - \ell_k(\hat{x}^{k+1}) &\geq& \frac{1}{2(1-\zeta)}(\hat{x}^{k+1} - x^k)^TH_k(\hat{x}^{k+1} - x^k) \nonumber \\
&\geq& \frac{\mu_k}{2(1-\zeta)}\|\hat{x}^{k+1} - x^k\|^2. \label{ineq:l_k-decrease}
\end{eqnarray}
Due to the convexity of $\ell_k$, for any $\alpha\in[0,1]$, we have
$$ \ell_k(x^k) - \ell_k(x^k + \alpha d^k) \geq \alpha\left( \ell_k(x^k) - \ell_k(x^k + d^k) \right). $$
Combining the above two inequalities and noting that $\hat{x}^{k+1} = x^k + d^k$, we obtain
\begin{equation}\label{eq:ell-lb}
\ell_k(x^k) - \ell_k(x^k + \alpha d^k) \geq \frac{\alpha\mu_k}{2(1-\zeta)}\|d^k\|^2. 
\end{equation}
Moreover, since $\nabla f$ is Lipschitz continuous by Assumption \ref{ass:smoothness}\ref{ass:L1}, we have
$$
f(x^k + \alpha d^k) - f(x^k) \leq \alpha\nabla f(x^k)^T d^k + \frac{\alpha^2L_1}{2}\|d^k\|^2, 
$$
which, by the definition of $\ell_k$, leads to
$$ F(x^k) - F(x^k + \alpha d^k) \geq \ell_k(x^k) - \ell_k(x^k + \alpha d^k) - \frac{\alpha^2L_1}{2}\|d^k\|^2. $$
%
It then follows from the above inequality and \eqref{eq:ell-lb} that
\begin{equation*}
\begin{aligned}
& F(x^k) - F(x^k + \alpha d^k) - \theta \left( \ell_k(x^k) - \ell_k(x^k + \alpha d^k) \right) \\
\geq\,& (1-\theta) \left( \ell_k(x^k) - \ell_k(x^k + \alpha d^k) \right) - \frac{\alpha^2L_1}{2}\|d^k\|^2 \\
\geq\,& \frac{(1-\theta)\alpha\mu_k}{2(1-\zeta)}\|d^k\|^2 - \frac{\alpha^2L_1}{2}\|d^k\|^2 \\
=\,& \frac{\alpha}{2}\left(\frac{1-\theta}{1-\zeta}\mu_k - \alpha L_1\right)\|d^k\|^2.
\end{aligned}
\end{equation*}
Hence, as long as $\alpha$ satisfies $\alpha < (1-\theta)\mu_k/(1-\zeta)L_1$, the descent condition~\eqref{eq:descent-cond} is satisfied. Since the backtracking line search multiplies the step length by $\beta\in(0,1)$ after each trial, the line search strategy will output an $\alpha_k$ that satisfies $\alpha_k\geq \beta(1-\theta)\mu_k/(1-\zeta)L_1$ in a finite number of steps.
\end{proof}
\noindent Combining Lemmas~\ref{lem:exist-hatx} and~\ref{lem:exist-step-size}, we conclude that \textsf{IRPN} is well defined.

\section{Convergence Analysis of \textsf{IRPN}}\label{sec:convergence-analysis}
In this section, we establish the global convergence and local convergence rate of \textsf{IRPN}. Let $\{x^k\}_{k\geq 0}$ be the sequence of iterates generated by Algorithm \ref{alg:IRPN}. We first present the following result, which shows that the residues $\{r(x^k)\}_{k\ge0}$ eventually vanish.
\begin{prop}\label{prop:global-converge}
Suppose that Assumption \ref{ass:smoothness}\ref{ass:L1} holds. Then, we have
$$
\lim_{k\rightarrow\infty} r(x^k) = 0.
$$
In particular, every accumulation point of the sequence $\{x^k\}_{k\geq 0}$ is an optimal solution to problem~\eqref{eq:basic-prob}.
\end{prop}

\begin{proof}
From~\eqref{eq:ell-lb}, we obtain
$$
\ell_k(x^k) - \ell_k(x^k + \alpha d^k) \geq \frac{\alpha\mu_k}{2(1-\zeta)}\|d^k\|^2 \geq 0. 
$$
Hence, due to the descent condition \eqref{eq:descent-cond} and the assumption that $F^* > -\infty$, we have $\lim_{k\rightarrow\infty} \left( \ell_k(x^k) - \ell_k(x^k + \alpha_k d^k) \right) = 0$.
This, together with~\eqref{eq:ell-lb} again, implies that 
\begin{equation} \label{eq:res-limit-1}
\lim_{k\rightarrow\infty}\alpha_k\mu_k\|d^k\|^2 = 0.
\end{equation}
On the other hand, we have $r_k(\hat{x}^{k+1}) \leq \eta\cdot r(x^k)$ from the inexactness condition~\eqref{eq:inexact-cond}. It then follows that
\begin{eqnarray}
(1-\eta)r(x^k) &\leq& r(x^k) - r_k(\hat{x}^{k+1}) \nonumber \\
&=& r_k(x^k) - r_k(\hat{x}^{k+1}) \nonumber \\
&\leq& (\lambda_{\max}(H_k) + 2)\|d^k\| \nonumber \\
&\leq& (L_1 + \mu_k + 2)\|d^k\|, \label{eq:res-limit-2}
\end{eqnarray}
where the first equality follows from $r_k(x^k) = r(x^k)$, the second inequality is by Corollary \ref{cor:lip-resi-map-im}, and the last inequality is due to Fact~\ref{fact:bound-hess}. By combining~\eqref{eq:res-limit-1} and~\eqref{eq:res-limit-2} with Lemma~\ref{lem:exist-step-size}, we obtain
$$ \lim_{k\rightarrow\infty} \frac{\mu_k^2}{(L_1 + \mu_k + 2)^2}r(x^k)^2 = 0. $$
Since $\mu_k = c\cdot r(x^k)^\rho$, we conclude that $\lim_{k\rightarrow\infty} r(x^k) = 0$. 
\end{proof}
Although Proposition~\ref{prop:global-converge} and the Luo-Tseng EB property~\eqref{eq:def-eb} together imply that $\mbox{dist}(x^k,\mathcal{X}) \rightarrow 0$, the latter does not guarantee the convergence of the sequence $\{x^k\}_{k\ge0}$. Hence, some extra arguments are needed to establish the global convergence of \textsf{IRPN}. As it turns out, by invoking the Luo-Tseng EB property~\eqref{eq:def-eb}, we can bound the local rate at which the sequence $\{\mbox{dist}(x^k,\mathcal{X})\}_{k\ge0}$ tends to zero, from which we can establish not only the convergence but also the local rate of convergence of the sequence $\{x^k\}_{k\ge0}$. To begin, let us present the following result, whose proof can be found in Section~\ref{sec:pf-thm-3}.

\begin{prop} \label{prop:convergence-rate}
Suppose that Assumptions~\ref{ass:smoothness} and~\ref{ass:ass-eb} hold. Then, for sufficiently large $k$, we have
\begin{enumerate}[wide = 0pt, labelwidth = 2em, labelsep*=0em, itemindent = 0pt, leftmargin = \dimexpr\labelwidth + \labelsep\relax, noitemsep,topsep = 1ex, font=\normalfont, label=(\roman*)]
\item ${\rm dist}(x^{k+1},\mathcal{X})\leq \gamma \cdot {\rm dist}(x^k,\mathcal{X})$ for some $\gamma\in(0,1)$ if we take $\rho = 0$, $c\leq \tfrac{\kappa}{4}$, and $\eta\leq \tfrac{1}{2(L_1+3)^2}$;
\item ${\rm dist}(x^{k+1},\mathcal{X})\leq O( {\rm dist}(x^k,\mathcal{X})^{1+\rho})$ if we take $\rho\in(0,1)$;
\item ${\rm dist}(x^{k+1},\mathcal{X})\leq O( {\rm dist}(x^k,\mathcal{X})^{2})$ if we take $\rho =1$, $c\geq \kappa L_2$, and $\eta\leq \tfrac{\kappa^2 L_2}{2(L_1+1)}$.
\end{enumerate}
\end{prop}

As remarked previously, the setting of Proposition~\ref{prop:convergence-rate} precludes the use of techniques that are based on either the positive definiteness of the Hessian of $f$ at an optimal solution or the uniqueness of the optimal solution to prove it. Our techniques, which heavily exploit the Luo-Tseng EB property~\eqref{eq:def-eb}, are different from those in \cite{LSS14,BNO16} and should find further applications in the convergence rate analysis of other second-order methods in the absence of strong convexity.

Using Proposition~\ref{prop:convergence-rate}, we can then prove that the sequence of iterates generated by \textsf{IRPN} converges to an optimal solution to problem~\eqref{eq:basic-prob} and establish the local rate of convergence.


\begin{thm}[Global Convergence and Local Rate of Convergence of \textsf{IRPN}]\label{thm:seq_conv}
Suppose that Assumptions~\ref{ass:smoothness} and~\ref{ass:ass-eb} hold. Then, in all three cases of Proposition~\ref{prop:convergence-rate}, the sequence $\{x^k\}_{k\geq 0}$ converges to some $x^*\in\mathcal{X}$. Moreover, the convergence rate is at least
\begin{enumerate}[wide = 0pt, labelwidth = 2em, labelsep*=0em, itemindent = 0pt, leftmargin = \dimexpr\labelwidth + \labelsep\relax, noitemsep,topsep = 1ex, font=\normalfont, label=(\roman*)]
\item \emph{$R$-linear} if $\rho=0$, $c \leq\frac{\kappa}{4}$, and $\eta\leq \frac{1}{2(L_1+2)^2}$;

\item \emph{$Q$-superlinear} with order $1 + \rho$ if $\rho\in(0,1)$;

\item \emph{$Q$-quadratic} if $\rho=1$, $c\geq \kappa L_2$, and $\eta\leq \frac{\kappa^2L_2}{2(L_1+1)}$.
\end{enumerate}
\end{thm}

Remarkably, Theorem~\ref{thm:seq_conv} shows that \textsf{IRPN} can attain a superlinear or even quadratic rate of convergence without strong convexity.  The proof of Theorem~\ref{thm:seq_conv} can be found in Section~\ref{sec:pf-cor-3}.

Although Theorem \ref{thm:seq_conv} shows that a larger $\rho$ will result in a faster convergence rate with respect to the outer iteration counter $k$, it does not necessarily mean that a larger $\rho$ will lead to a better overall complexity. The reason is that a larger $\rho$ results in a more stringent inexactness condition (see~\eqref{eq:inexact-cond}) and thus the time consumed by each iteration is longer. Hence, despite the faster convergence rate, the best choice of $\rho\in[0,1]$ depends on the problem and the inner solver (see also Section~\ref{sec:numerical}). Nonetheless, Theorem~\ref{thm:seq_conv} provides a complete characterization of the convergence rate of \textsf{IRPN} in terms of $\rho$ and facilitates the flexibility of \textsf{IRPN}. Developing an empirical or analytic approach to tuning $\rho$ is definitely a topic worth pursuing.

Since Proposition~\ref{prop:convergence-rate} plays a crucial role in obtaining our convergence results and its proof is rather tedious, let us give an overview of the proof here. Let $\bar{x}^k$ be the projection of $x^k$ onto the optimal solution set $\mathcal{X}$. Also, let $\tilde{x}^{k+1}$ and $\hat{x}^{k+1}$ be the exact and approximate minimizers of the quadratic model $q_k$, respectively; see Figure~\ref{fig:roadmap}. Our goal is to show that $\mbox{dist}(x^{k+1},\mathcal{X}) = O(\mbox{dist}(x^k,\mathcal{X})^{1+\rho})$. Towards that end, we first prove that the distances $\| \tilde{x}^{k+1} - x^k \|$ and $\| \tilde{x}^{k+1} - \hat{x}^{k+1} \|$ are comparable, in the sense that they are both of $O({\rm dist}(x^k,\mathcal{X}))$; see Lemmas~\ref{lem:ex-to-x^k} and~\ref{lem:ex-to-x^k+1}. This implies that $\| \hat{x}^{k+1} - x^k \|$ is also of $O({\rm dist}(x^k,\mathcal{X}))$. Then, we show that eventually the step sizes computed by Algorithm~\ref{alg:IRPN} will all equal 1, which means that $x^{k+1} = \hat{x}^{k+1}$ for all sufficiently large $k$; see Lemma~\ref{lem:asymptotic-unit-step-size}.  As a result, we have $\| x^{k+1} - x^k \| = O({\rm dist}(x^k,\mathcal{X}))$ for all sufficiently large $k$. Finally, we show that $r(x^{k+1}) = O(\| x^{k+1} - x^k \|^{1+\rho})$ and invoke the Luo-Tseng EB property~\eqref{eq:def-eb} to conclude that 
$$ {\rm dist}(x^{k+1},\mathcal{X}) \le \kappa \cdot r(x^{k+1}) = O( \| \hat{x}^{k+1} - x^k \|^{1+\rho}) = O({\rm dist}(x^k,\mathcal{X})^{1+\rho}). $$

\begin{figure}[htb]
\centering
\includegraphics[width = 0.45\linewidth]{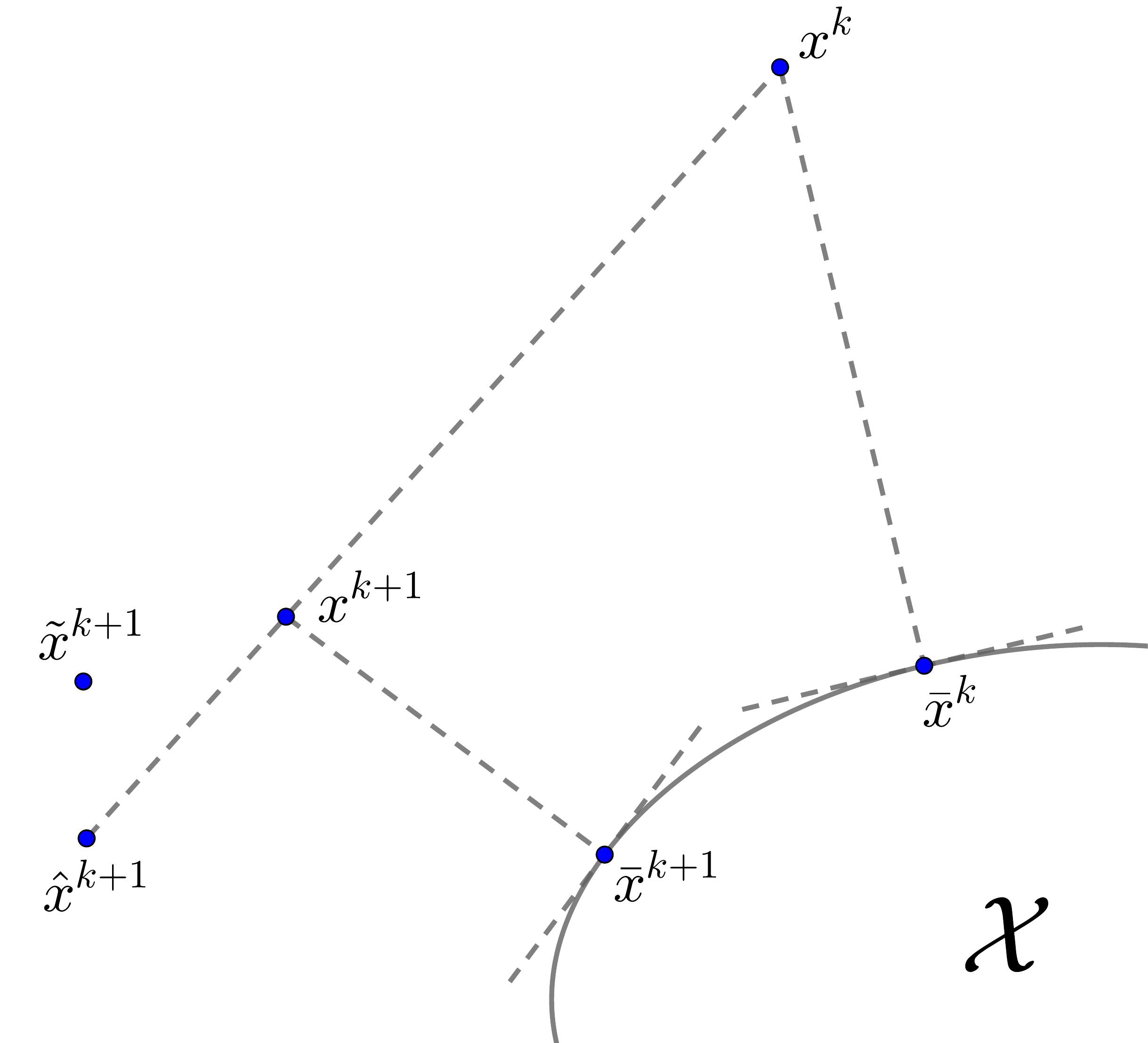}
\caption{Illustration of the geometry in the proof of Proposition~\ref{prop:convergence-rate}.}
\label{fig:roadmap}
\end{figure}

The rest of this section is devoted to proving Proposition~\ref{prop:convergence-rate} and Theorem~\ref{thm:seq_conv}. We first prove some technical lemmas in Section \ref{sec:tech_lem}. Then, we prove Proposition~\ref{prop:convergence-rate} and Theorem~\ref{thm:seq_conv} in Sections~\ref{sec:pf-thm-3} and \ref{sec:pf-cor-3}, respectively.

\subsection{Technical Lemmas}\label{sec:tech_lem}
Throughout this subsection, Assumptions \ref{ass:smoothness} and \ref{ass:ass-eb} are in force. Recall that $\bar{x}^k$ is the projection of $x^k$ onto the optimal solution set $\mathcal{X}$; $\tilde{x}^{k+1}$ and $\hat{x}^{k+1}$ are the exact and approximate minimizers of the quadratic model $q_k$, respectively.

\begin{lemma}\label{lem:exact}
It holds for all $k$ that
\begin{equation*}
\|\tilde{x}^{k+1}-\bar{x}^k\|\leq \frac{1}{\mu_k}\|\nabla f(\bar{x}^k) - \nabla f(x^k) - H_k(\bar{x}^k - x^k)\|.
\end{equation*}
\end{lemma}
\begin{proof}
If $\tilde{x}^{k+1}=\bar{x}^k$, then the inequality holds trivially. Hence, suppose that $\tilde{x}^{k+1}\neq\bar{x}^k$. As $\bar{x}^k\in \mathcal{X}$, we have $\mathbf{0}\in \nabla f(\bar{x}^k)+\partial g(\bar{x}^k)$. Moreover, the definition of $\tilde{x}^{k+1}$ yields $\mathbf{0}\in \nabla f(x^k)+H_k(\tilde{x}^{k+1}-x^k)+\partial g(\tilde{x}^{k+1})$. Using these and the monotonicity of the subdifferential mapping $\partial g$ of the closed proper convex function $g$, we obtain
\begin{eqnarray*}
 0  &\leq& \left\langle \nabla f(x^k) - \nabla f(\bar{x}^k) + H_k(\tilde{x}^{k+1}-x^k) , \bar{x}^k - \tilde{x}^{k+1} \right\rangle \\
 &=& \left\langle \nabla f(x^k) - \nabla f(\bar{x}^k) - H_k(x^k - \bar{x}^k) , \bar{x}^k - \tilde{x}^{k+1} \right\rangle \\
 & & \quad - \left\langle H_k(\bar{x}^k - \tilde{x}^{k+1}) , \bar{x}^k - \tilde{x}^{k+1} \right\rangle\\
& \leq & \| \nabla f(x^k) - \nabla f(\bar{x}^k) - H_k(x^k - \bar{x}^k)\| \cdot \|\bar{x}^k - \tilde{x}^{k+1} \| - \mu_k \|\bar{x}^k - \tilde{x}^{k+1}\|^2,
\end{eqnarray*}
which yields the desired inequality.
\end{proof}
\noindent We remark here that we did not use any special structure of $H_k$ in the proof of Lemma \ref{lem:exact} except the positive definiteness of $H_k$. Therefore, the same result will hold for any positive definite approximate Hessian $H_k$, with $\mu_k$ replaced by the minimum eigenvalue $\lambda_{\min}(H_k)$ of $H_k$.

\begin{lemma}\label{lem:ex-to-x^k}
It holds for all $k$ that
\begin{equation*}
\| \tilde{x}^{k+1} - x^k \| \leq \left(\frac{L_2}{2\mu_k}{\rm dist}(x^k,\mathcal{X}) + 2\right){\rm dist}(x^k,\mathcal{X}).
\end{equation*}
\end{lemma}
\begin{proof}
Using the triangle inequality, the definition of $\mu_k$, and Assumption \ref{ass:smoothness}\ref{ass:L2}, we compute
\begin{eqnarray*}
&& \| \nabla f(\bar{x}^k) - \nabla f(x^k) - H_k(\bar{x}^k - x^k) \| \nonumber \\
&\leq & \| \nabla f(\bar{x}^k) - \nabla f(x^k) - \nabla^2 f(x^k)(\bar{x}^k - x^k) \| + \mu_k\|\bar{x}^k - x^k\| \nonumber \\
&\leq & \frac{L_2}{2} \|\bar{x}^k - x^k\|^2 + \mu_k \|\bar{x}^k - x^k\|. \label{ineq:Hessian}
\end{eqnarray*}
This, together with Lemma \ref{lem:exact}, yields
\begin{eqnarray*}
\|\tilde{x}^{k+1} - x^k\| & \leq& \|\tilde{x}^{k+1} - \bar{x}^k\| + \mbox{dist}(x^k,\mathcal{X})\\
&\leq& \frac{1}{\mu_k}\|\nabla f(\bar{x}^k) - \nabla f(x^k) - H_k(\bar{x}^k - x^k)\| + \mbox{dist}(x^k,\mathcal{X})\\
& \leq& \frac{L_2}{2\mu_k} \|\bar{x}^k - x^k\|^2 +  \|\bar{x}^k - x^k\| + \mbox{dist}(x^k,\mathcal{X})\\
& =& \left(\frac{L_2}{2\mu_k}{\rm dist}(x^k,\mathcal{X}) + 2\right){\rm dist}(x^k,\mathcal{X}),
\end{eqnarray*}
as desired.
\end{proof}

\begin{lemma}\label{lem:ex-to-x^k+1}
It holds for all $k$ that
\begin{equation*}
\|\hat{x}^{k+1} - \tilde{x}^{k+1}\|\leq \frac{\eta(L_1+|1-\mu_k|)}{c}r(x^k)+ \eta \cdot r(x^k)^{1+\rho}.
\end{equation*}
\end{lemma}
\begin{proof}
The inequality holds trivially if $\hat{x}^{k+1}=\tilde{x}^{k+1}$. Hence, we assume that $\hat{x}^{k+1}\neq \tilde{x}^{k+1}$. Recall that $R_k(\hat{x}^{k+1})=\hat{x}^{k+1}-\mbox{prox}_g(\hat{x}^{k+1}-\nabla f_k(\hat{x}^{k+1}))$. Thus,
$$ R_k(\hat{x}^{k+1}) - \nabla f_k(\hat{x}^{k+1}) \in \partial g \left( \hat{x}^{k+1} - R_k(\hat{x}^{k+1}) \right). $$
Since $q_k=f_k+g$, we have
$$ R_k(\hat{x}^{k+1}) + \nabla f_k \left( \hat{x}^{k+1} - R_k(\hat{x}^{k+1}) \right) - \nabla f_k(\hat{x}^{k+1}) \in \partial q_k \left( \hat{x}^{k+1} - R_k(\hat{x}^{k+1}) \right). $$
Using the fact that $\nabla f_k(x) = \nabla f(x^k) + H_k(x-x^k)$, the above leads to
$$ (I-H_k)R_k(\hat{x}^{k+1}) \in \partial q_k(\hat{x}^{k+1} - R_k(\hat{x}^{k+1})). $$
On the other hand, we have $\mathbf{0}\in \partial q_k(\tilde{x}^{k+1})$. Since $q_k$ is $\mu_k$-strongly convex, the subdifferential mapping $\partial q_k$ is $\mu_k$-strongly monotone and thus
$$
\left\langle (I-H_k)R_k(\hat{x}^{k+1}) , \hat{x}^{k+1} - R_k(\hat{x}^{k+1}) -\tilde{x}^{k+1} \right\rangle
\ge  \mu_k \|\hat{x}^{k+1} - R_k(\hat{x}^{k+1}) -\tilde{x}^{k+1}\|^2.
$$
Upon applying the Cauchy-Schwarz inequality to the above, we obtain
\begin{eqnarray*}
\|\hat{x}^{k+1} - R_k(\hat{x}^{k+1}) -\tilde{x}^{k+1}\| &\leq& \frac{1}{\mu_k} \left\| (I-H_k)R_k(\hat{x}^{k+1}) \right\| \\
&\leq& \frac{1}{\mu_k} \|\nabla^2 f(x^k)-(1-\mu_k)I\| \cdot r_k(\hat{x}^{k+1})\\
&\leq& \frac{\eta(L_1+|1-\mu_k|)}{\mu_k}r(x^k)^{1+\rho}\\
&\leq& \frac{\eta(L_1+|1-\mu_k|)}{c}r(x^k),
\end{eqnarray*}
where the third inequality is due to Fact~\ref{fact:bound-hess} and the inexactness condition~\eqref{eq:inexact-cond}, and the last is due to the definition of $\mu_k$. Hence, we obtain
\begin{eqnarray*}
\|\hat{x}^{k+1} - \tilde{x}^{k+1}\| &\leq& \|\hat{x}^{k+1} - R_k(\hat{x}^{k+1}) -\tilde{x}^{k+1}\| + \|R_k(\hat{x}^{k+1})\| \\
&\leq& \frac{\eta(L_1+|1-\mu_k|)}{c}r(x^k)+ \eta \cdot r(x^k)^{1+\rho},
\end{eqnarray*}
as desired.
\end{proof}

Next, we show that Algorithm~\ref{alg:IRPN} eventually takes a unit step size; \ie, $\alpha_k=1$ and $x^{k+1}=x^k+d^k$ for all sufficiently large $k$. 
\begin{lemma}\label{lem:asymptotic-unit-step-size}
Suppose that Assumptions \ref{ass:smoothness} and \ref{ass:ass-eb} hold. Then, there exists an integer $k_0 \ge 0$ such that $\alpha_k=1$ for all $k\geq k_0$ in either of the following cases:
\begin{enumerate}[wide = 0pt, labelwidth = 2em, labelsep*=0em, itemindent = 0pt, leftmargin = \dimexpr\labelwidth + \labelsep\relax, noitemsep,topsep = 1ex, font=\normalfont, label=(\roman*)]
\item $\rho\in [0,1)$.
\item $\rho=1$ and $c,\eta$ satisfy $2 \eta L_2(L_1+2) + \kappa^2 L_2^2 + 2c\kappa L_2 \leq  6 c^2$.
\end{enumerate}
The inequality in (ii) is satisfied if we take $c\geq \kappa L_2$ and $\eta \leq \tfrac{\kappa^2 L_2}{2(L_1+1)}$.
\end{lemma}
\begin{proof}
By the Fundamental Theorem of Calculus and Assumption~\ref{ass:smoothness}\ref{ass:L2}, we compute
\begin{eqnarray*}
& & f(\hat{x}^{k+1}) - f(x^k) \\
&=& \int_0^1 \nabla f \left( x^k + t(\hat{x}^{k+1} - x^k) \right)^T ( \hat{x}^{k+1} - x^k ) \, dt\\
&=& \int_0^1 \left[ \nabla f \left( x^k + t(\hat{x}^{k+1} - x^k) \right) - \nabla f(x^k) \right]^T ( \hat{x}^{k+1} - x^k ) \,dt \\
& & \quad+ \nabla f(x^k)^T ( \hat{x}^{k+1} - x^k ) \\
&=& \int_0^1 \int_0^1 t (\hat{x}^{k+1} - x^k)^T\left[ \nabla^2 f \left( x^k + st(\hat{x}^{k+1} - x^k) \right) - \nabla^2 f(x^k) \right] (\hat{x}^{k+1}-x^k) \, ds\,dt\\
& & \quad + \int_0^1 t (\hat{x}^{k+1}-x^k)^T\nabla^2 f(x^k) (\hat{x}^{k+1}-x^k) \, dt + \nabla f(x^k)^T \left( \hat{x}^{k+1} - x^k \right) \\
&\leq& \frac{L_2}{6}\|\hat{x}^{k+1} - x^k\|^3 + \frac{1}{2}(\hat{x}^{k+1}-x^k)^T\nabla^2 f(x^k)(\hat{x}^{k+1}-x^k) + \nabla f(x^k)^T(\hat{x}^{k+1} - x^k).
\end{eqnarray*}
Therefore, we have
\begin{eqnarray}
& & F(\hat{x}^{k+1}) - F(x^k) + q_k(x^k) - q_k(\hat{x}^{k+1}) \nonumber \\
&=& f(\hat{x}^{k+1}) - f(x^k) - \nabla f(x^k)^T(\hat{x}^{k+1}-x^k) \nonumber \\
& & \quad - \frac{1}{2}(\hat{x}^{k+1}-x^k)^T \nabla^2 f(x^k)(\hat{x}^{k+1}-x^k) - \frac{\mu_k}{2}\|\hat{x}^{k+1}-x^k\|^2 \nonumber \\
&\leq& \frac{L_2}{6}\|\hat{x}^{k+1} - x^k\|^3 - \frac{\mu_k}{2}\|\hat{x}^{k+1}-x^k\|^2. \label{ineq:drop}
\end{eqnarray}
It then follows from~\eqref{eq:inexact-cond-quad},~\eqref{ineq:l_k-decrease}, and~\eqref{ineq:drop} that
\begin{eqnarray*}
& & F(\hat{x}^{k+1})-F(x^k) \\
&=& \left( F(\hat{x}^{k+1})-F(x^k) + q_k(x^k) - q_k(\hat{x}^{k+1}) \right) + \left( q_k(\hat{x}^{k+1}) - q_k(x^k) \right)\\
&\leq& \frac{L_2}{6}\|d^k\|^3 - \frac{\mu_k}{2}\|d^k\|^2 + \zeta\left( \ell_k(\hat{x}^{k+1}) - \ell_k(x^k) \right) \\
&=& \frac{L_2}{6}\|d^k\|^3 - \frac{\mu_k}{2}\|d^k\|^2 + \theta\left( \ell_k(\hat{x}^{k+1}) - \ell_k(x^k) \right) + (\zeta - \theta)\left( \ell_k(\hat{x}^{k+1}) - \ell_k(x^k) \right) \\
&\leq& \frac{L_2}{6}\|d^k\|^3 + \theta\left( \ell_k(\hat{x}^{k+1}) - \ell_k(x^k) \right) - \left(1 + \frac{\zeta - \theta}{1-\zeta}\right)\frac{\mu_k}{2}\|d^k\|^2.
\end{eqnarray*}
In particular, if 
\begin{equation}\label{ineq:unit-step-condition}
\|d^k\|=\|\hat{x}^{k+1} - x^k\| \leq \left(1 + \frac{\zeta - \theta}{1-\zeta}\right)\frac{3c \cdot r(x^k)^\rho}{L_2},
\end{equation}
then $F(\hat{x}^{k+1})-F(x^k) \le \theta\left( \ell_k(\hat{x}^{k+1}) - \ell_k(x^k) \right)$, which, according to the line search strategy in Algorithm~\ref{alg:IRPN}, would lead to a unit step size $\alpha_k=1$. Hence, it remains to determine for what choices of the parameters $c,\eta$ would inequality~\eqref{ineq:unit-step-condition} hold. By Lemmas~\ref{lem:ex-to-x^k} and~\ref{lem:ex-to-x^k+1}, for all sufficiently large $k$,
\begin{eqnarray*}
\|\hat{x}^{k+1} - x^k\| & \leq & \frac{\eta(L_1 + |1 - \mu_k|)}{c}r(x^k)+ \eta \cdot r(x^k)^{1+\rho}\\
& & \quad + \left(\frac{L_2}{2c \cdot r(x^k)^\rho}\mbox{dist}(x^k,\mathcal{X}) + 2\right)\mbox{dist}(x^k,\mathcal{X})\\
& \leq & \frac{\eta(L_1 + |1 - \mu_k|)}{c}r(x^k)+ \eta \cdot r(x^k)^{1+\rho} \\
& & \quad + \frac{\kappa^2 L_2}{2c}r(x^k)^{2-\rho} + 2\kappa \cdot r(x^k).
\end{eqnarray*}
Hence, a sufficient condition for inequality~\eqref{ineq:unit-step-condition} to hold is
\begin{eqnarray}
&& \frac{\eta(L_1 + |1 - \mu_k|)}{c}r(x^k)+ \eta \cdot r(x^k)^{1+\rho}  + \frac{\kappa^2 L_2}{2c}r(x^k)^{2-\rho} + 2\kappa \cdot r(x^k) \nonumber \\
& \leq & \left(1 + \frac{\zeta - \theta}{1-\zeta}\right)\frac{3c \cdot r(x^k)^\rho}{L_2}. \label{ineq:c-condition}
\end{eqnarray}
If $\rho\in [0,1)$, then since $r(x^k)\rightarrow 0$, we see that inequality~\eqref{ineq:c-condition} holds for all sufficiently large $k$. On the other hand, if $\rho = 1$, then since $\zeta>\theta$ and $\mu_k\in(0,1)$ for all sufficiently large $k$, inequality~\eqref{ineq:c-condition} holds for all sufficiently large $k$ if 
\begin{align*}
\frac{\eta (L_1+1)}{c} +\frac{\kappa^2 L_2}{2c} + 2\kappa & \leq \frac{3c}{L_2},
\end{align*}
which can be satisfied by taking $c\geq \kappa L_2$ and $\eta\leq \tfrac{\kappa^2 L_2}{2(L_1+1)}$.
\end{proof}

\subsection{Proof of Proposition \ref{prop:convergence-rate}}\label{sec:pf-thm-3}
By Lemma \ref{lem:asymptotic-unit-step-size}, in all three cases of Proposition~\ref{prop:convergence-rate}, we have $\alpha_k=1$ and hence $\hat{x}^{k+1}=x^{k+1}$ for all sufficiently large $k$. Consequently, all the results derived in Section~\ref{sec:tech_lem} will continue to hold if we replace $\hat{x}^{k+1}$ by $x^{k+1}$ when $k$ is sufficiently large. From Proposition \ref{prop:lip-resi-map}, we have
\begin{eqnarray}
\|x^{k+1} - x^k\| &\leq& \frac{\eta(L_1 + |1 - \mu_k|)}{c}r(x^k)+ \eta \cdot r(x^k)^{1+\rho} \nonumber \\
& & \quad+ \left(\frac{L_2}{2c\cdot r(x^k)^\rho} {\rm dist}(x^k,\mathcal{X}) + 2\right){\rm dist}(x^k,\mathcal{X}) \nonumber \\
&\leq& \Bigg(\frac{\eta(L_1 + |1 - \mu_k|)(L_1+2)}{c}+ \eta (L_1+2) r(x^k)^{\rho} \nonumber \\
& & \quad + \frac{L_2}{2c \cdot r(x^k)^\rho}{\rm dist}(x^k,\mathcal{X}) + 2\Bigg){\rm dist}(x^k,\mathcal{X}).  \label{ineq:x^k-x^k+1-2}
\end{eqnarray}
By the Luo-Tseng EB property~\eqref{eq:def-eb}, we have ${\rm dist}(x,\mathcal{X})\leq \kappa \cdot r(x)$ whenever $r(x)$ is sufficiently small. Since $r(x^k)\rightarrow 0$, we see from~\eqref{ineq:x^k-x^k+1-2} that for all sufficiently large $k$, 
\begin{equation*}
\|x^{k+1} - x^k\| = O(\mbox{dist}(x^k,\mathcal{X})).
\end{equation*}
Using the non-expansiveness of ${\rm prox}_g$ and the fact that $\nabla f_k(x^{k+1}) = \nabla f(x^k) + H_k(x^{k+1}-x^k)$, we have
\begin{eqnarray}
&& \| \mbox{prox}_g(x^{k+1} - \nabla f(x^{k+1})) - \mbox{prox}_g(x^{k+1} - \nabla f_k(x^{k+1}))\| \nonumber \\
& \leq & \| \nabla f(x^{k+1}) - \nabla f(x^k) - H_k(x^{k+1} - x^k) \|. \label{ineq:prox}
\end{eqnarray}
It follows that for all sufficiently large $k$,
\begin{eqnarray}
{\rm dist}(x^{k+1},\mathcal{X}) & \leq&  \kappa \cdot r(x^{k+1}) \nonumber \\
& \leq & \kappa \left( \|R(x^{k+1}) - R_k(x^{k+1})\|+ \|R_k(x^{k+1})\| \right) \nonumber \\
& \leq & \kappa \| \mbox{prox}_g(x^{k+1} - \nabla f(x^{k+1})) - \mbox{prox}_g(x^{k+1} - \nabla f_k(x^{k+1}))\| \nonumber \\
& & \quad + \kappa \eta \cdot r(x^k)^{1+\rho} \nonumber \\
& \leq & \kappa\| \nabla f(x^{k+1}) - \nabla f(x^k) - H_k(x^{k+1} - x^k) \| + \kappa\eta \cdot r(x^k)^{1+\rho} \nonumber \\
& \leq & \frac{\kappa L_2}{2}\| x^{k+1}- x^k \|^2 + \kappa \mu_k \|x^{k+1}- x^k\|  + \kappa\eta\cdot r(x^k)^{1+\rho} \nonumber \\
& \leq & \frac{\kappa L_2}{2}\| x^{k+1}- x^k \|^2 + c\kappa (L_1+2)^\rho {\rm dist}(x^k,\mathcal{X})^\rho \|x^{k+1}- x^k\| \nonumber \\
& & \quad + \kappa\eta (L_1+2)^{1+\rho} \mbox{dist}(x^k,\mathcal{X})^{1+\rho}\label{eq:lin-conv-cond} \\
& \leq & O({\rm dist}(x^k,\mathcal{X})^{1+\rho}), \label{ineq:proof-main}
\end{eqnarray}
where the third inequality is due to the inexactness condition~\eqref{eq:inexact-cond}, the fourth follows from~\eqref{ineq:prox}, the sixth is due to the definition of $\mu_k$ and Proposition \ref{prop:lip-resi-map}, and the last follows from~\eqref{ineq:x^k-x^k+1-2}.

For the case where $\rho=0$, in order to establish the linear convergence of the sequence $\{{\rm dist}(x^k,\mathcal{X})\}_{k\ge0}$, we need the constant in the big-O notation in~\eqref{ineq:proof-main} to be strictly less than 1. Upon inspecting~\eqref{ineq:x^k-x^k+1-2} and~\eqref{eq:lin-conv-cond}, we see that this can be guaranteed if
$$
\kappa\eta (L_1+2) + c\kappa \left(\frac{\eta (L_1 + |1 - c|)(L_1+2)}{c} + \eta (L_1+2) + 2 \right) < 1.
$$
The above inequality can be satisfied by taking $c\leq \tfrac{\kappa}{4}$ and $\eta \leq \tfrac{1}{2(L_1+3)^2}$. \qed

\subsection{Proof of Theorem \ref{thm:seq_conv}}\label{sec:pf-cor-3}
By the Luo-Tseng EB property~\eqref{eq:def-eb} and Proposition~\ref{prop:global-converge}, we have ${\rm dist}(x^k,\mathcal{X}) \rightarrow 0$. This implies that in all three cases of Proposition~\ref{prop:convergence-rate}, there exists an integer $K_1\ge0$ such that ${\rm dist}(x^{k+1},\mathcal{X})\leq \gamma\cdot{\rm dist}(x^k,\mathcal{X})$ for all $k\geq K_1$, where $\gamma\in(0,1)$ is the constant in case (i) of Proposition~\ref{prop:convergence-rate}. In addition, by~\eqref{ineq:x^k-x^k+1-2}, there exist a real number $\sigma>0$ and an integer $K_2\ge0$ such that $\|x^{k+1} - x^k\|\leq \sigma\cdot\mbox{dist}(x^k,\mathcal{X})$ for all $k\geq K_2$. Lastly, since ${\rm dist}(x^k,\mathcal{X}) \rightarrow 0$, given any $\epsilon>0$, there exists an integer $K_3\ge0$ such that ${\rm dist}(x^k,\mathcal{X})\leq (1-\gamma)\epsilon/\sigma$ for all $k\geq K_3$. Upon taking $K = \max\{K_1,K_2,K_3\}$, we have
\begin{equation*}
\begin{split}
\|x^{k_1} - x^{k_0}\| & \leq \sum_{i=k_0}^{k_1-1}\|x^{i+1} - x^i\| \leq \sigma \sum_{i=k_0}^{k_1 - 1}\mbox{dist}(x^i,\mathcal{X}) \\ 
& \leq \sigma \sum_{i=0}^{\infty}\gamma^i\cdot\mbox{dist}(x^{k_0},\mathcal{X}) \leq \frac{\sigma}{1 - \gamma}\cdot\mbox{dist}(x^{k_0},\mathcal{X}) \leq \epsilon
\end{split}
\end{equation*}
for all $k_1\geq k_0\geq K$. It follows that $\{x^k\}_{k\geq 0}$ is a Cauchy sequence and hence converges to some $x^{*}$. This, together with ${\rm dist}(x^k,\mathcal{X}) \rightarrow 0$ and the fact that $\mathcal{X}$ is closed, implies that $x^{*}\in\mathcal{X}$. 

Now, it remains to establish the convergence rate of the sequence $\{x^k\}_{k\geq 0}$ in cases (i), (ii), and (iii). From the above analysis, it is immediate that 
$$\| x^{k+i} - x^{k+1}\| \leq \frac{\sigma}{1-\gamma}\cdot{\rm dist}(x^{k+1},\mathcal{X}) $$
for all integers $i\geq 1$ and $k\geq K$. Upon taking $i\rightarrow\infty$, we have
\begin{equation}
\label{eq:step-seq-conv}
\|x^{k+1} - x^{*}\| \leq \frac{\sigma}{1-\gamma}\cdot{\rm dist}(x^{k+1},\mathcal{X}).
\end{equation}
In case (i), Proposition~\ref{prop:convergence-rate} implies that the sequence $\{\mbox{dist}(x^{k+1},\mathcal{X})\}_{k\ge0}$ converges $R$-linearly to $0$. Hence, by \eqref{eq:step-seq-conv}, the sequence $\{x^k\}_{k\geq 0}$ converges $R$-linearly to $x^{*}$. In case (ii), using Proposition~\ref{prop:convergence-rate} and \eqref{eq:step-seq-conv}, we obtain 
$$ \|x^{k+1} - x^{*}\| = O(\mbox{dist}(x^k,\mathcal{X})^{1+\rho}) \leq O(\|x^k - x^{*}\|^{1+\rho}), $$
which implies that the sequence $\{x^k\}_{k\geq 0}$ converges $Q$-superlinearly to $x^{*}$ with order $1 + \rho$. The same arguments show that the sequence $\{x^k\}_{k\geq 0}$ converges $Q$-quadratically to $x^{*}$ in case (iii). This completes the proof. \qed

\section{Numerical Experiments}\label{sec:numerical}
In this section, we study the numerical performance of our proposed \textsf{IRPN} and compare it with some existing algorithms. We focus on the $\ell_1$-regularized logistic regression problem, which takes the following form:
\begin{equation}\label{eq:ell_1-logreg}
\min_{x\in\mathbb{R}^n} \Bigg\{ F(x) = \underbrace{\frac{1}{m}\sum_{i=1}^m \log\left(1+\exp(-b_i\cdot a_i^Tx)\right)}_{f(x)} + \underbrace{\lambda\|x\|_1}_{g(x)} \Bigg\}. 
\end{equation}
Here, $a_1,\ldots,a_m\in\mathbb{R}^n$ are given data samples; $b_1,\ldots,b_m\in\{1,-1\}$ are given labels; $\lambda>0$ is a given regularization parameter. Problem~\eqref{eq:ell_1-logreg} arises in linear classification tasks in machine learning and is a standard benchmark problem for testing the efficiency of different algorithms for solving problem~\eqref{eq:basic-prob}. 
In our experiments, we use the data sets \textsf{colon-cancer}, \textsf{rcv1}, and \textsf{news20} from~\cite{LIBSVM-Data}
and set $\lambda = 5\times 10^{-4}$. The sizes of and the numbers of non-zero entries in these data sets are listed in Table \ref{tab:data}. Since $m<n$ in all three data sets, the corresponding objective functions $F$ in~\eqref{eq:ell_1-logreg} are not strongly convex. However, due to Fact \ref{fact:eb-holds}, problem~\eqref{eq:ell_1-logreg} possesses the Luo-Tseng EB property~\eqref{eq:def-eb}. Hence, both Assumptions \ref{ass:smoothness} and \ref{ass:ass-eb} are satisfied by problem~\eqref{eq:ell_1-logreg} and our convergence analysis of \textsf{IRPN} applies.

\renewcommand{\arraystretch}{1.5}
\begin{table}[h]
	\centering
	\caption{Tested data sets.}
	\label{tab:data}
	\begin{tabular}{c c c c}
		\hline
		Data set & $n$ & $m$ & \# of nnz \\ 
		\hline \hline
		\textsf{colon-cancer} & 2000 & 62 & 124000 \\ \hline
		\textsf{rcv1} & 47236 & 20242 & 1498952 \\ \hline
		\textsf{news20} & 1355191 & 19996 & 9097916 \\
		\hline
	\end{tabular}
\end{table}


All experiments are coded in MATLAB (R2017b) and run on a Dell desktop with a 3.50-GHz Intel Core E3-1270 v3 processor and 32 GB of RAM.\footnote{The code can be downloaded from \url{https://github.com/ZiruiZhou/IRPN}.} We next present the list of tested algorithms and discuss their implementation details below.

\textsf{FISTA}: The description of this algorithm can be found in \cite{beck2009fast}. In our experiments, we use the constant step size $1/L_1$, where $L_1 = \|A\|^2/(4m)$ and $A=[a_1 \, \cdots \, a_m]\in\mathbb{R}^{n\times m}$.\footnote{$\|A\|^2$ is computed via the MATLAB code \textsf{lambda = eigs(A*A',1,'LM')}.} It can be verified that $L_1$ is the Lipschitz constant of $\nabla f$ in~\eqref{eq:ell_1-logreg}. In addition, we restart the algorithm if $(y^{k-1} - x^{k})^T(x^{k} - x^{k-1})>0$ for some $k$, where $y^{k-1}$ is the extrapolation point in iteration $k$. Such restarting strategy has been found to be empirically efficient \cite{OC15,WCP17} and it indeed improves the performance of \textsf{FISTA} for solving problem~\eqref{eq:ell_1-logreg} in our experiments.

\textsf{SpaRSA}: This algorithm is a non-monotone proximal gradient method, whose description can be found in \cite{WNF09}. Following the notations in \cite{WNF09}, we set in our experiments $\sigma = 10^{-4}$, $\alpha_{\min} = 10^{-8}$, $\alpha_{\max} = 10^8$, $M = 5$, $\alpha_0 = 1$, and
$$ \alpha_k = \max\left\{\alpha_{\min}, \min\left\{\alpha_{\max}, \frac{\|\Delta x\|^2}{|\Delta G^T\Delta x|}\right\}\right\} $$
for $k\geq 1$, where $\Delta x = x^k - x^{k-1}$ and $\Delta G = \nabla f(x^k) - \nabla f(x^{k-1})$. 

\textsf{CGD}: This algorithm is the coordinate gradient descent method, whose general description can be found in \cite{tseng2009coordinate}. It has been successfully applied to solve problem \eqref{eq:ell_1-logreg}; see, \eg,~\cite{YCHL10,YT11}. Given the current iterate $x^k$, we randomly choose a coordinate index $i_k\in\{1,\ldots,n\}$ and compute a descent direction $d^k\in\mathbb{R}^n$ by letting $d^k_i = 0$ for all $i\neq i_k$ and
$$ d^k_{i_k} = \argmin_{s\in\mathbb{R}} \left\{ \frac{1}{2}H_{i_ki_k}s^2 + g_{i_k}s + \lambda|x_{i_k}^k + s| \right\}, $$
where $H_{i_ki_k} = [\nabla^2 f(x^k)]_{i_ki_k} + \nu$ with $\nu = 10^{-6}$ and $g_{i_k} = [\nabla f(x^k)]_{i_k}$. Then, we choose $\alpha_k$ to be the largest element in $\{\beta^j\}_{j\ge0}$ satisfying 
$$ F(x^k + \alpha_k d^k) \leq F(x^k) + \sigma\alpha_k\left( \nabla f(x^k)^Td^k + \lambda \|x^k + \alpha_k d^k\|_1 - \lambda\|x^k\|_1\right), $$
where we set $\beta = 0.25$ and $\sigma = 0.5$ in our experiments, and set $x^{k+1} = x^k + \alpha_kd^k$. To make the number of iterations of \textsf{CGD} comparable to those of \textsf{FISTA} and \textsf{SpaRSA}, we refer to every $n$ coordinate updates as an iteration of \textsf{CGD}. In addition, while the operations in each iteration of \textsf{FISTA} and \textsf{SpaRSA} can be vectorized in MATLAB, those of \textsf{CGD} cannot and require a {\em for} loop, which is inefficient in MATLAB. To remedy this issue, we implement \textsf{CGD} in MATLAB as a C source MEX-file.

\textsf{IRPN}: This is the family of inexact SQA methods proposed in this paper, whose description can be found in Algorithm \ref{alg:IRPN}. In our experiments, we set $\theta = \beta = 0.25$, $\zeta = 0.4$, $c = 10^{-6}$, and $\eta = 0.5$. We test the three values $0$, $0.5$, and $1$ for $\rho$, which, according to Theorem \ref{thm:seq_conv}, result in the local linear, superlinear, and quadratic convergence of \textsf{IRPN}, respectively. Notice that in the context of problem~\eqref{eq:ell_1-logreg}, each inner problem of \textsf{IRPN} is to minimize the sum of a strongly convex quadratic function and the non-smooth convex function $\lambda\|x\|_1$. Hence, it can be solved using the coordinate descent algorithm with simple updates; see, \eg, \cite{YHL12}. Due to the same reason as \textsf{CGD}, we code the inner solver of \textsf{IRPN} in MATLAB as a C source MEX-file. Moreover, we refer to every $n$ coordinate updates as an iteration of the inner solver; \ie, an inner iteration of \textsf{IRPN}. 

\textsf{newGLMNET}: The description of this algorithm can be found in \cite{YHL12}. In our experiments, we implement it as a special case of \textsf{IRPN}. In particular, we set $c = 10^{-12}$ as suggested in \cite{YHL12} and $\rho = 0$ in Algorithm \ref{alg:IRPN}. Moreover, instead of the stopping criterion~\eqref{eq:inexact-cond}, \textsf{newGLMNET} uses a heuristic stopping criterion for the inner problem~\cite{YHL12}. Specifically, in the $k$-th iteration, the inner solver is terminated when $r_k(\hat{x}^{k+1})\leq \epsilon_{\rm in}$, where $\epsilon_{\rm in}=0.1$ initially and $\epsilon_{\rm in} = \epsilon_{\rm in}/4$ if the stopping criterion is satisfied after only one inner iteration.

\smallskip
It is worth mentioning the computation of the Hessian-vector product for the logistic loss function, which is explained in detail in \cite{YHL12} and is crucial to the success of Newton-type methods for solving problem~\eqref{eq:ell_1-logreg}. For the logistic loss function $f$ in \eqref{eq:ell_1-logreg},
one can verify by a direct computation that $\nabla^2 f(x) = ADA^T$, where $A=[a_1 \, \cdots \, a_m]\in\mathbb{R}^{n\times m}$ and $D$ is a diagonal matrix with
$$ D_{ii} = \frac{\exp(-b_i\cdot a_i^Tx)}{\left(\exp(-b_i\cdot a_i^Tx) + 1\right)^2}, \quad i=1,\ldots,m. $$
Notice that $A$ is sparse in many real data sets such as \textsf{colon-cancer}, \textsf{rcv1}, and \textsf{news20}. Hence, $\nabla^2 f(x)v$ for any $v\in\mathbb{R}^n$ can be evaluated efficiently in our experiments.

We initialize all the tested algorithms by the same point. In order to see the impact of different initial points on the performance of these algorithms, we choose $x^0 = 0$ and $x^0 = 10\xi$, where the entries of $\xi$ are sampled {\em i.i.d.} from the standard normal distribution. All the algorithms are terminated if the iterate $x^k$ satisfies $r(x^k)\leq \epsilon_0$, where $\epsilon_0 = 10^{-4}$, $10^{-6}$, and $10^{-8}$ in our experiments. 

The computational results on the data sets \textsf{colon-cancer}, \textsf{rcv1}, and \textsf{news20} are presented in Tables \ref{tab:cc}, \ref{tab:rcv1}, and \ref{tab:new20}, respectively. Each table has two sub-tables, which show the results with the two initial points $x^0 = 0$ and $x^0 = 10\xi$, respectively. In each sub-table, we record the number of outer iterations, the number of inner iterations, and the CPU time (in seconds) of all the tested algorithms for achieving the required accuracy. We note that \textsf{FISTA} is the only one among these algorithms that requires the Lipschitz constant $L_1$. The time for computing $L_1$ is added to the time of \textsf{FISTA} and can be found at the bottom of each table. From Tables \ref{tab:cc}, \ref{tab:rcv1}, and~\ref{tab:new20}, we can see that the SQA methods (\textsf{newGLMNET} and \textsf{IRPN}) perform the best among all the tested algorithms. In particular, we have the following observations:
\begin{enumerate}[wide = 0pt, labelwidth = 2em, labelsep*=0em, itemindent = 0pt, leftmargin = \dimexpr\labelwidth + \labelsep\relax, noitemsep,topsep = 1ex, font=\normalfont, label={(S\arabic*)}]
	\resetspb
	\item[\subpb] The SQA methods substantially outperform \textsf{FISTA}. \smallskip
	\item[\subpb] \textsf{SpaRSA} and \textsf{CGD} are competitive with the SQA methods when $\epsilon_0 = 10^{-4}$, while the SQA methods outperform \textsf{SpaRSA} and \textsf{CGD} when $\epsilon_0 = 10^{-6}$ and $10^{-8}$. \smallskip
	\item[\subpb] As the parameter $\rho$ in \textsf{IRPN} increases from $0$ to $1$, the number of outer iterations needed by \textsf{IRPN} decreases, while the number of inner iterations within each outer iteration increases. The overall performance of $\rho=0$ is competitive with that of $\rho = 0.5$ and better than that of $\rho = 1$.
\end{enumerate}
It is worth noting that the number of inner iterations needed by \textsf{CGD} is consistently the least among the tested algorithms. However, since it requires a line search procedure after every coordinate update, its overall performance is worse than that of the SQA methods, which only require one line search procedure in each outer iteration. Such drawback of \textsf{CGD} is also noted in \cite{YHL12}.

\setlength{\tabcolsep}{6pt}
\renewcommand{\arraystretch}{1.5}
\begin{table}[t]
	\centering
	\caption{Computational results on \textsf{colon-cancer}.}
	\label{tab:cc}
	\begin{subtable}[t]{\linewidth}
		\centering
		\caption{$x^0 = 0$}
		\label{stab:cc1}
		\begin{tabular}{c c c  c  c c c c c}
			\hline 
			\multirow{2}{*}{Tol.} & \multirow{2}{*}{Index} & \multirow{2}{*}{\textsf{FISTA}$^*$} & \multirow{2}{*}{\textsf{SpaRSA}} & \multirow{2}{*}{\textsf{CGD}} & \multirow{2}{*}{\textsf{newG.}} & \textsf{IRPN} & \textsf{IRPN} & \textsf{IRPN} \\
			& & & & & & $\rho = 0$ & $\rho = 0.5$ & $\rho = 1$ \\
			\hline \hline
			\multirow{3}{*}{$10^{-4}$} & outer iter. & -- & -- & -- & 10 & 6 & 4 & 4 \\
			& inner iter. & 904 & 104 & 18 & 33 & 26 & 37 & 87 \\
			& time & 0.41 & 0.03 & 0.08 & 0.03 & \textbf{0.02} & 0.03 & 0.06 \\
			\hline 
			\multirow{3}{*}{$10^{-6}$} & outer iter. & -- & -- & -- & 19 & 14 & 5 & 5 \\
			& inner iter. & 3856 & 397 & 55 & 92 & 84 & 85 & 183 \\
			& time & 1.78 & 0.13 & 0.32 & \textbf{0.06} & \textbf{0.06} & \textbf{0.06} & 0.13 \\
			\hline
			\multirow{3}{*}{$10^{-8}$} & outer iter. & -- & -- & -- & 25 & 24 & 6 & 6 \\
			& inner iter. & 5202 & 756 & 100 & 144 & 162 & 142 & 273 \\
			& time & 2.30 & 0.23 & 0.40  & \textbf{0.08} & 0.14 & 0.10 & 0.18 \\
			\hline
		\end{tabular}
	\end{subtable}
	
	\bigskip
	\bigskip
	\begin{subtable}[t]{\linewidth}
		\centering
		\caption{$x^0 = 10\xi$}
		\label{stab:cc2}
		\begin{tabular}{c c c c c c c c c}
			\hline 
			\multirow{2}{*}{Tol.} & \multirow{2}{*}{Index} & \multirow{2}{*}{\textsf{FISTA}$^*$} & \multirow{2}{*}{\textsf{SpaRSA}} & \multirow{2}{*}{\textsf{CGD}} & \multirow{2}{*}{\textsf{newG.}} & \textsf{IRPN} & \textsf{IRPN} & \textsf{IRPN} \\
			& & & & & & $\rho = 0$ & $\rho = 0.5$ & $\rho = 1$ \\
			\hline \hline
			\multirow{3}{*}{$10^{-4}$} & outer iter. & -- & -- & -- & 14 & 11 & 10 & 7 \\
			& inner iter. & 1185 & 124 & 33 & 39 & 39 & 110 & 126 \\
			& time & 0.43 & 0.04 & 0.13 & \textbf{0.03} & \textbf{0.03} & 0.09 & 0.19 \\
			\hline
			\multirow{3}{*}{$10^{-6}$} & outer iter. & -- & -- & -- & 25 & 17 & 9 & 8 \\
			& inner iter. & 3536 & 422 & 71 & 101 & 85 & 113 & 268 \\
			& time & 1.50 & 0.12 & 0.32 & 0.08 & \textbf{0.06} & 0.09 & 0.27 \\
			\hline
			\multirow{3}{*}{$10^{-8}$} & outer iter. & -- & -- & -- & 25 & 24 & 6 & 6 \\
			& inner iter. & 4932 & 727 & 100 & 173 & 175 & 118 & 373 \\
			& time & 2.18 & 0.25 & 0.43  & 0.11 & 0.13 & \textbf{0.08} & 0.30 \\
			\hline
		\end{tabular}
		
		\medskip
		* The time for computing the Lipschitz constant $L_1$ in \textsf{FISTA} is 0.04.
	\end{subtable}
\end{table}

\begin{table}[t]
	\centering
	\caption{Computational results on \textsf{rcv1}.}
	\label{tab:rcv1}
	\begin{subtable}[t]{\linewidth}
		\centering
		\caption{$x^0 = 0$}
		\label{stab:rcv11}
		\begin{tabular}{c c c c c c c c c}
			\hline 
			\multirow{2}{*}{Tol.} & \multirow{2}{*}{Index} & \multirow{2}{*}{\textsf{FISTA}$^*$} & \multirow{2}{*}{\textsf{SpaRSA}} & \multirow{2}{*}{\textsf{CGD}} & \multirow{2}{*}{\textsf{newG.}} & \textsf{IRPN} & \textsf{IRPN} & \textsf{IRPN} \\
			& & & & & & $\rho = 0$ & $\rho = 0.5$ & $\rho = 1$ \\
			\hline
			\hline
			\multirow{3}{*}{$10^{-4}$} & outer iter. & -- & -- & -- & 9 & 5 & 3 & 3 \\
			& inner iter. & 113 & 30 & 8 & 20 & 19 & 20 & 34 \\
			& time & 39.13 & 0.29 & 0.46 & 0.32 & 0.27 & \textbf{0.26} & 0.42 \\
			\hline
			\multirow{3}{*}{$10^{-6}$} & outer iter. & -- & -- & -- & 17 & 11 & 5 & 5 \\
			& inner iter. & 377 & 280 & 35 & 55 & 48 & 36 & 106 \\
			& time & 45.42 & 2.60 & 2.05 & 0.77 & 0.64 & \textbf{0.48} & 1.21 \\
			\hline
			\multirow{3}{*}{$10^{-8}$} & outer iter. & -- & -- & -- & 23 & 16 & 5 & 5 \\
			& inner iter. & 697 & 549 & 61 & 90 & 77 & 61 & 126 \\
			& time & 46.67 & 5.08 & 3.51  & 1.18 & 1.03 & \textbf{0.70} & 1.32 \\
			\hline
		\end{tabular}
	\end{subtable}
	
	\bigskip
	\bigskip
	\begin{subtable}[t]{\linewidth}
		\centering
		\caption{$x^0 = 10\xi$}
		\label{stab:rcv12}
		\begin{tabular}{c c c c c c c c c}
			\hline
			\multirow{2}{*}{Tol.} & \multirow{2}{*}{Index} & \multirow{2}{*}{\textsf{FISTA}$^*$} & \multirow{2}{*}{\textsf{SpaRSA}} & \multirow{2}{*}{\textsf{CGD}} & \multirow{2}{*}{\textsf{newG.}} & \textsf{IRPN} & \textsf{IRPN} & \textsf{IRPN} \\
			& & & & & & $\rho = 0$ & $\rho = 0.5$ & $\rho = 1$ \\
			\hline
			\hline
			\multirow{3}{*}{$10^{-4}$} & outer iter. & -- & -- & -- & 12 & 8 & 5 & 5 \\
			& inner iter. & 308 & 50 & 15 & 25 & 23 & 27 & 44 \\
			& time & 41.58 & 0.48 & 0.88 & 0.41 & \textbf{0.35} & 0.37 & 0.57 \\
			\hline
			\multirow{3}{*}{$10^{-6}$} & outer iter. & -- & -- & -- & 19 & 12 & 7 & 7 \\
			& inner iter. & 595 & 286 & 47 & 54 & 47 & 57 & 112 \\
			& time & 45.25 & 2.65 & 2.77 & 0.77 & \textbf{0.66} & 0.68 & 1.24 \\
			\hline
			\multirow{3}{*}{$10^{-8}$} & outer iter. & -- & -- & -- & 26 & 19 & 8 & 8 \\
			& inner iter. & 1009 & 591 & 67 & 101 & 87 & 97 & 141 \\
			& time & 51.31 & 5.47 & 3.85  & 1.34 & 1.17 & \textbf{1.14} & 1.54 \\
			\hline
		\end{tabular}
	\end{subtable}
	
	\medskip
	* The time for computing the Lipschitz constant $L_1$ in \textsf{FISTA} is 36.14.
\end{table}

\begin{table}[t]
	\centering
	\caption{Computational results on \textsf{news20}.}
	\label{tab:new20}
	\begin{subtable}[t]{\linewidth}
		\centering
		\caption{$x^0 = 0$}
		\label{stab:news201}
		\begin{tabular}{c c c c c c c c c}
			\hline 
			\multirow{2}{*}{Tol.} & \multirow{2}{*}{Index} & \multirow{2}{*}{\textsf{FISTA}$^*$} & \multirow{2}{*}{\textsf{SpaRSA}} & \multirow{2}{*}{\textsf{CGD}} & \multirow{2}{*}{\textsf{newG.}} & \textsf{IRPN} & \textsf{IRPN} & \textsf{IRPN} \\
			& & & & & & $\rho = 0$ & $\rho = 0.5$ & $\rho = 1$ \\
			\hline
			\hline
			\multirow{3}{*}{$10^{-4}$} & outer iter. & -- & -- & -- & 9 & 5 & 3 & 3 \\
			& inner iter. & 182 & 48 & 33 & 41 & 37 & 54 & 127 \\
			& time & 77.68 & 5.63 & 12.13 & 4.92 & \textbf{4.16} & 5.68 & 12.87 \\
			\hline
			\multirow{3}{*}{$10^{-6}$} & outer iter. & -- & -- & -- & 17 & 12 & 4 & 4 \\
			& inner iter. & 566 & 220 & 98 & 115 & 100 & 107 & 223 \\
			& time & 131.01 & 25.31 & 36.21 & 12.58 & \textbf{10.99} & \textbf{10.99} & 22.40 \\
			\hline
			\multirow{3}{*}{$10^{-8}$} & outer iter. & -- & -- & -- & 23 & 18 & 4 & 4 \\
			& inner iter. & 978 & 580 & 155 & 189 & 170 & 157 & 248 \\
			& time & 188.77 & 67.62 & 57.23  & 20.15 & 18.51 & \textbf{16.36} & 24.84 \\
			\hline
		\end{tabular}
	\end{subtable}
	
	\bigskip
	\bigskip
	\begin{subtable}[t]{\linewidth}
		\centering
		\caption{$x^0 = 10\xi$}
		\label{stab:news202}
		\begin{tabular}{c c c c c c c c c}
			\hline 
			\multirow{2}{*}{Tol.} & \multirow{2}{*}{Index} & \multirow{2}{*}{\textsf{FISTA}$^*$} & \multirow{2}{*}{\textsf{SpaRSA}} & \multirow{2}{*}{\textsf{CGD}} & \multirow{2}{*}{\textsf{newG.}} & \textsf{IRPN} & \textsf{IRPN} & \textsf{IRPN} \\
			& & & & & & $\rho = 0$ & $\rho = 0.5$ & $\rho = 1$ \\
			\hline
			\hline
			\multirow{3}{*}{$10^{-4}$} & outer iter. & -- & -- & -- & 13 & 7 & 6 & 5 \\
			& inner iter. & 647 & 296 & 57 & 76 & 67 & 89 & 95 \\
			& time & 150.79 & 35.31 & 21.10 & 7.74 & \textbf{7.40} & 9.06 & 9.60 \\
			\hline
			\multirow{3}{*}{$10^{-6}$} & outer iter. & -- & -- & -- & 20 & 14 & 7 & 6 \\
			& inner iter. & 1055 & 302 & 108 & 166 & 154 & 155 & 179 \\
			& time & 207.24 & 35.88 & 40.04 & 16.41 & 16.17 & \textbf{15.41} & 17.54 \\
			\hline
			\multirow{3}{*}{$10^{-8}$} & outer iter. & -- & -- & -- & 26 & 21 & 7 & 7 \\
			& inner iter. & 1430 & 946 & 182 & 212 & 189 & 191 & 286 \\
			& time & 258.89 & 108.36 & 67.34  & 21.03 & 19.72 & \textbf{18.72} & 27.77 \\
			\hline
		\end{tabular}
	\end{subtable}
	
	\medskip
	* The time for computing the Lipschitz constant $L_1$ in \textsf{FISTA} is 51.85.
\end{table}

Before we end this section, we present in Figure \ref{fig:conv_iter_time} the convergence behavior of \textsf{IRPN} with $\rho\in\{0,0.5,1\}$ for the data sets \textsf{rcv1} and \textsf{news20}. The left two sub-figures show the convergence of the sequence of residues $\{r(x^k)\}_{k\ge0}$ against the number of outer iterations, where $\{x^k\}_{k\ge0}$ is the sequence of iterates generated by \textsf{IRPN}. It can be observed that the sequence $\{r(x^k)\}_{k\ge0}$ exhibits a linear convergence rate when $\rho=0$ and a superlinear rate when $\rho = 0.5$ and $1$. This, together with the Luo-Tseng EB property~\eqref{eq:def-eb}, implies that the sequence $\{\mbox{dist}(x^k,\mathcal{X})\}_{k\ge0}$ converges at least linearly when $\rho = 0$ and at least superlinearly when $\rho = 0.5$ and $1$, which corroborates our results in Proposition~\ref{prop:convergence-rate}. Nonetheless, a faster convergence rate of the outer iterates does not imply a better overall performance of \textsf{IRPN}, as a larger $\rho$ leads to a more stringent inexactness condition (see~\eqref{eq:inexact-cond}) and hence requires more inner iterations within each outer iteration. With that being said, we present the right two sub-figures in Figure \ref{fig:conv_iter_time} for comparing the overall performance of \textsf{IRPN} with $\rho\in\{0,0.5,1\}$. In particular, we show the convergence behavior of the sequence of residues $\{r(x^k)\}_{k\ge0}$ against the CPU time used, where $\{x^k\}_{k\ge0}$ is the sequence of \emph{all} iterates (including both the inner and outer ones) generated by \textsf{IRPN}. It can be observed that in terms of overall performance, \textsf{IRPN} with $\rho = 0$ and $0.5$ are comparable to each other, and they both outperform \textsf{IRPN} with $\rho = 1$.

\begin{figure}[h!]
	\centering
	\begin{subfigure}[b]{.48\linewidth}
		\includegraphics[scale=0.42]{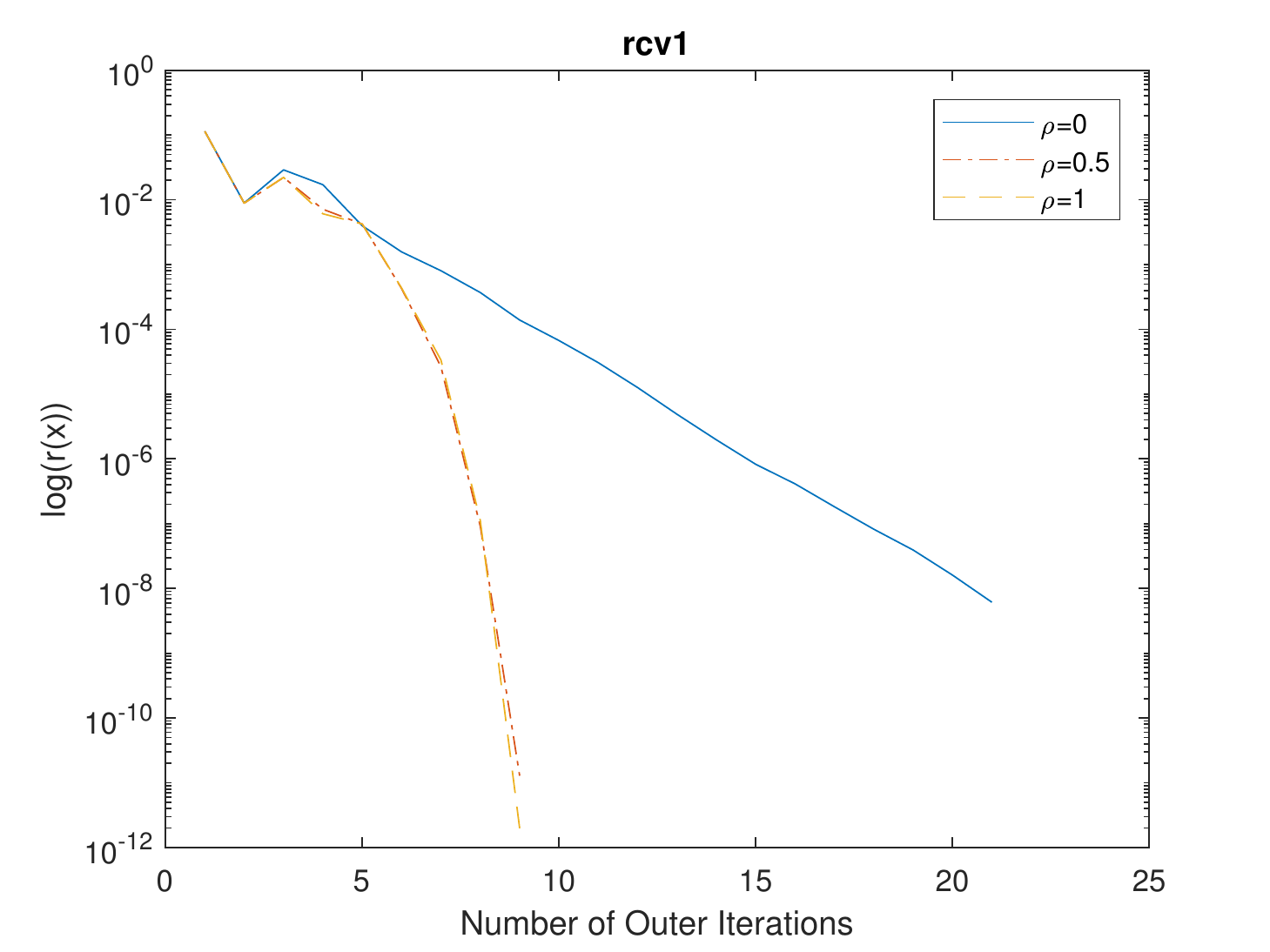}
	\end{subfigure}
	\hfill
	\begin{subfigure}[b]{.48\linewidth}
		\includegraphics[scale=0.42]{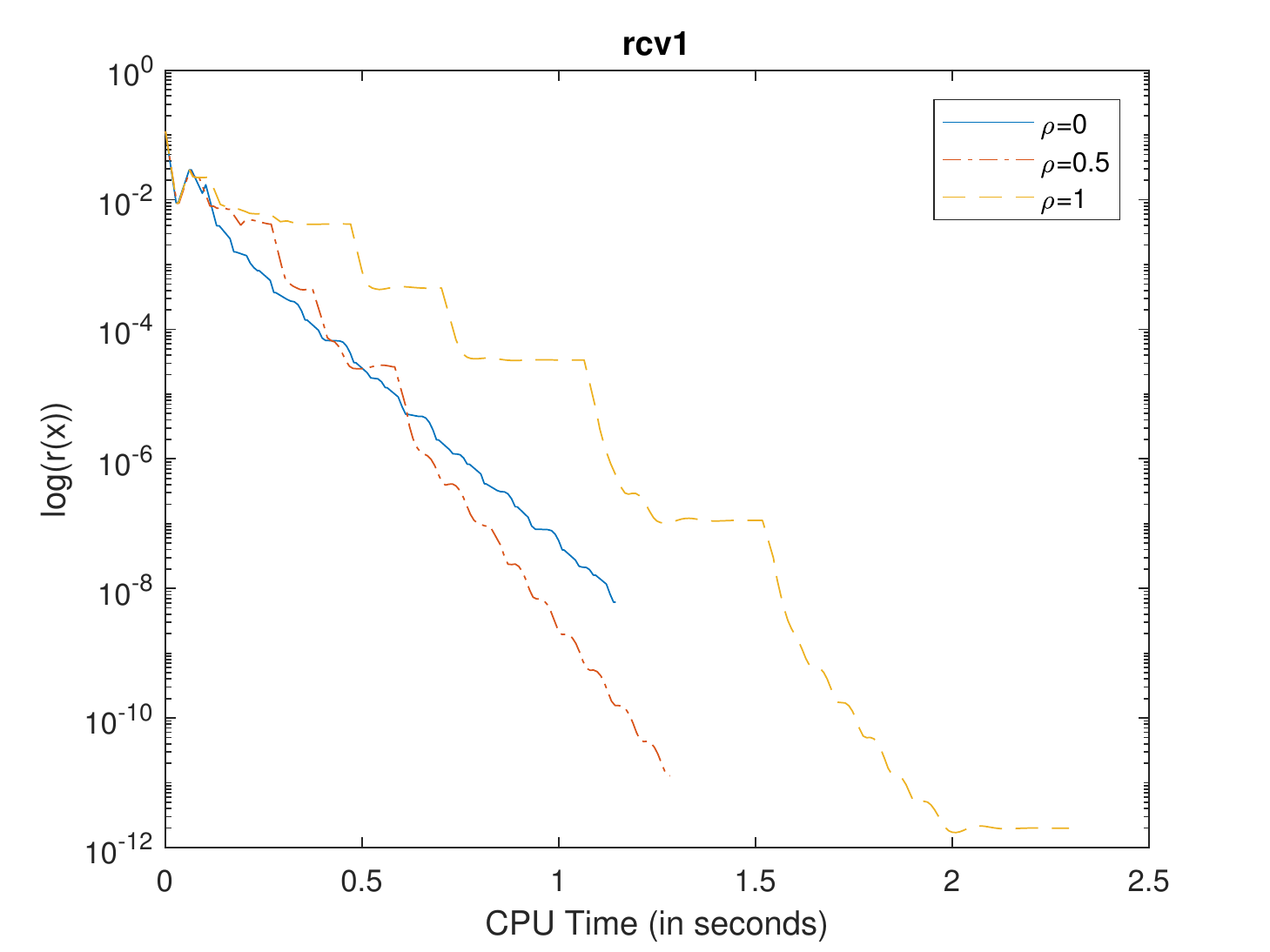}
	\end{subfigure}
	
	\medskip
	\begin{subfigure}[b]{.48\linewidth}
		\includegraphics[scale=0.42]{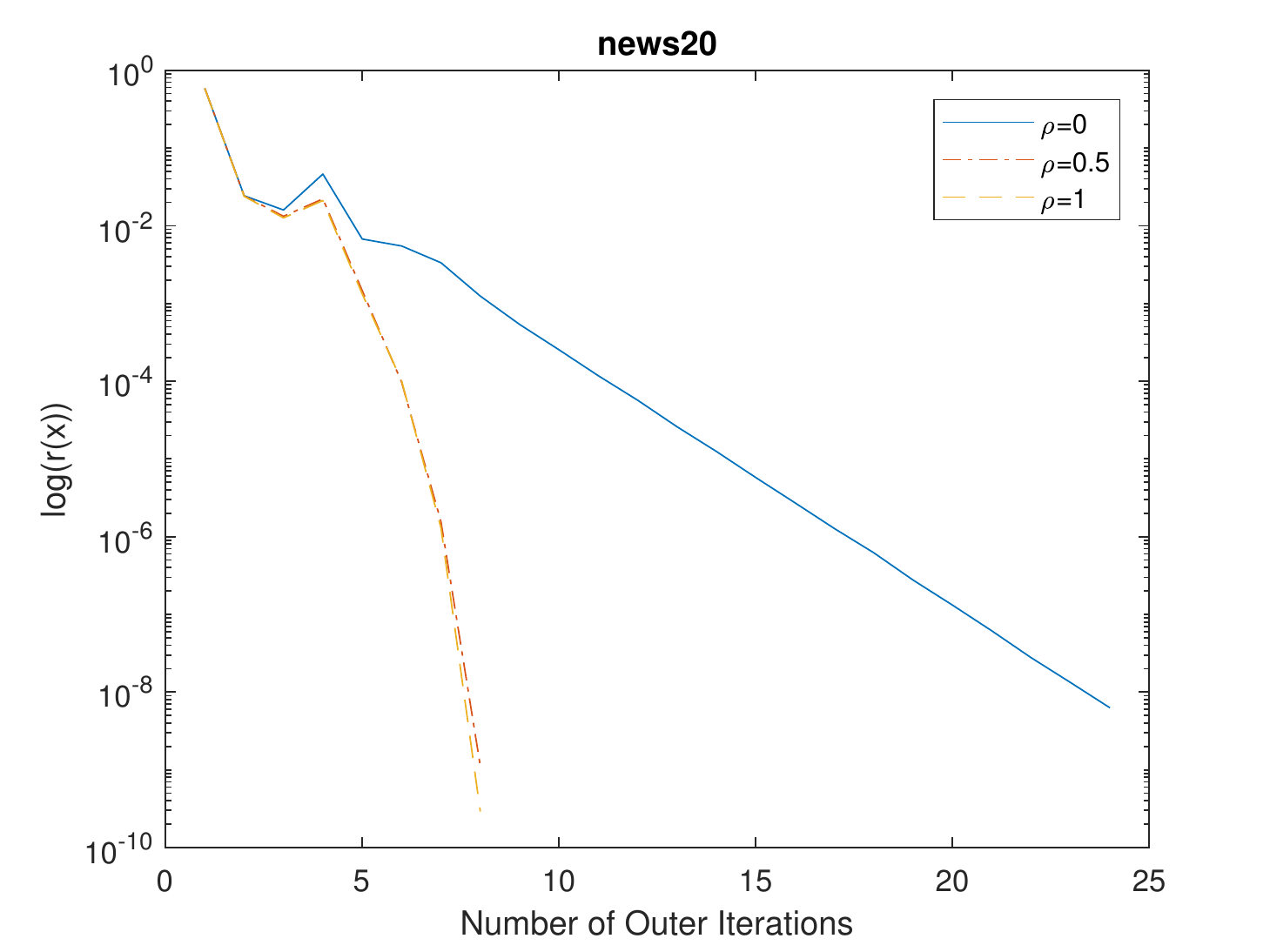}
	\end{subfigure}
	\hfill
	\begin{subfigure}[b]{.48\linewidth}
		\includegraphics[scale=0.42]{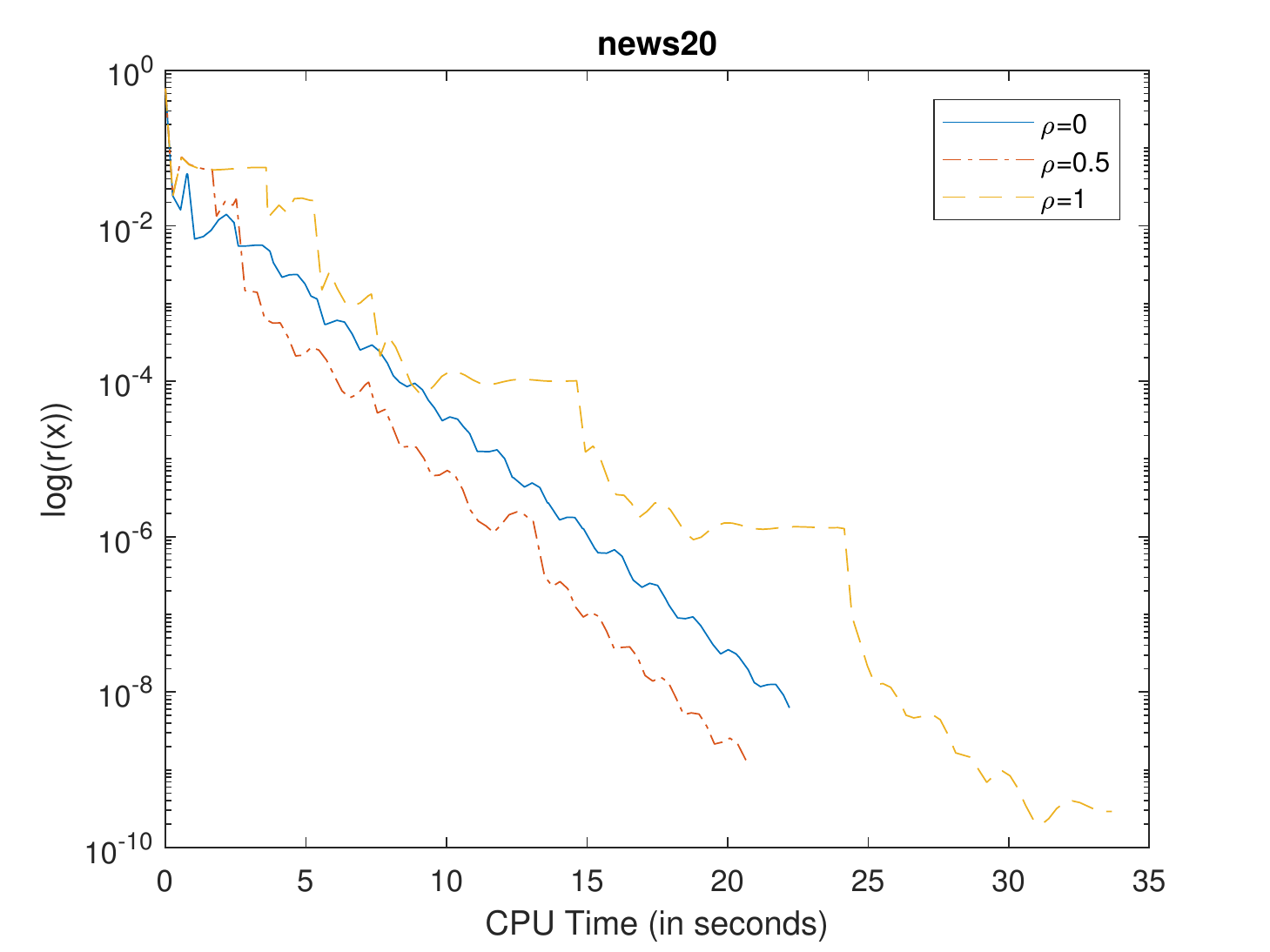}
	\end{subfigure}
	\caption{Convergence behavior of \textsf{IRPN} with $\rho \in \{0,0.5,1\}$.}
	\label{fig:conv_iter_time}
\end{figure}
 
\section{Conclusions}\label{sec:conclusion}
In this paper, we proposed a new family of inexact SQA methods called the \emph{inexact regularized proximal Newton} (\textsf{IRPN}) method for minimizing the sum of two closed proper convex functions, one of which is smooth and the other is possibly non-smooth. Compared with some prior SQA methods, \textsf{IRPN} is more flexible and has a much more satisfying convergence theory.  In particular, \textsf{IRPN} does not require an exact inner solver or the strong convexity of the smooth part of the objective function. Moreover, it can be shown to converge globally to an optimal solution and the local convergence rate is linear, superlinear, or even quadratic, depending on the choice of parameters of the algorithm. The key to our analysis is the Luo-Tseng error bound property. Although such property has played a fundamental role in establishing the linear convergence of various first-order methods, to the best of our knowledge, this is the first work to exploit such property to establish the superlinear convergence of SQA-type methods for non-smooth convex minimization. We compared our proposed \textsf{IRPN} with several popular and efficient algorithms by applying them to the $\ell_1$-regularized logistic regression problem. Experiment results indicate that \textsf{IRPN} achieves the desired levels of accuracy in shorter training time and less number of iterations than the other algorithms. This shows that \textsf{IRPN} not only has strong theoretical guarantees but also a superior numerical performance.

\begin{acknowledgements}
We thank the anonymous reviewers for their detailed and helpful comments. Most of the work of the first and second authors was done when they were PhD students at the Department of Systems Engineering and Engineering Management of The Chinese University of Hong Kong. 
\end{acknowledgements}

\bibliographystyle{abbrv}
\bibliography{sdpbib}

\end{document}